\documentclass[12pt]{amsart}
\usepackage{amsmath,amsthm,amsfonts,amscd,amssymb,eucal,latexsym,mathrsfs, appendix, yhmath, setspace}
\usepackage[numbers,sort&compress]{natbib}
\usepackage[all,cmtip]{xy}
\usepackage{enumerate}

\newtheorem{theorem}{Theorem}[section]
\newtheorem{corollary}[theorem]{Corollary}
\newtheorem{lemma}[theorem]{Lemma}
\newtheorem{proposition}[theorem]{Proposition}

\theoremstyle{definition}
\newtheorem{definition}[theorem]{Definition}
\newtheorem{remark}[theorem]{Remark}

\newtheorem{example}[theorem]{Example}
\newtheorem{question}[theorem]{Question}

\newcommand{\im}{{\rm im}}
\newcommand{\Aut}{{\rm Aut}}

\newcommand{\Ad}{{\rm Ad}\,}

\newcommand{\id}{{\rm id}}

\newcommand{\cB}{{\mathcal B}}

\newcommand{\cE}{{\mathcal E}}
\newcommand{\cF}{{\mathcal F}}

\newcommand{\cM}{{\mathcal M}}
\newcommand{\cN}{{\mathcal N}}

\newcommand{\cU}{{\mathcal U}}
\newcommand{\cV}{{\mathcal V}}
\newcommand{\cW}{{\mathcal W}}

\newcommand{\Cb}{{\mathbb C}}

\newcommand{\Eb}{{\mathbb E}}
\newcommand{\Fb}{{\mathbb F}}

\newcommand{\Kb}{{\mathbb K}}

\newcommand{\Nb}{{\mathbb N}}

\newcommand{\Pb}{{\mathbb P}}

\newcommand{\Rb}{{\mathbb R}}

\newcommand{\Zb}{{\mathbb Z}}
\newcommand{\sA}{{\mathscr A}}

\newcommand{\sV}{{\mathscr V}}

\newcommand{\rL}{{\rm L}}

\newcommand{\tr}{{\rm tr}}

\newcommand{\rk}{{\rm rk}}

\allowdisplaybreaks

\begin{document}

\title{Sylvester rank functions for amenable normal extensions}

\author{Baojie Jiang}
\address{\hskip-\parindent
Baojie Jiang,
College of Mathematics and Statistics, Chongqing University, Chongqing 401331, China.}
\email{jiangbaojie@gmail.com}

\author{Hanfeng Li}
\address{\hskip-\parindent
Hanfeng Li,
Center of Mathematics, Chongqing University,
Chongqing 401331, China and
 Department of Mathematics, SUNY at Buffalo,
Buffalo, NY 14260-2900, USA}
\email{hfli@math.buffalo.edu}

\date{November 20, 2020}


\begin{abstract}
We introduce a notion of amenable normal extension $S$ of a unital ring $R$ with a finite approximation system $\cF$, encompassing the amenable algebras over  a field of Gromov and Elek, the twisted crossed product by an amenable group, and the tensor product with a field extension. It is shown that every Sylvester matrix rank function $\rk$ of $R$ preserved by $S$ has a canonical extension to a Sylvester matrix rank function $\rk_{\cF}$ for $S$. In the case of twisted crossed product by an amenable group, and the tensor product with a field extension, it is also shown that $\rk_{\cF}$ depends on $\rk$ continuously. When an amenable group has a twisted action on a unital $C^*$-algebra preserving a tracial state, we also show that two natural Sylvester matrix rank functions on the algebraic twisted crossed product  constructed out of the tracial state coincide.
\end{abstract}

\maketitle

\tableofcontents

\section{Introduction} \label{S-introduction}

Sylvester rank functions for a given unital ring $R$ are numerical invariants for matrices or modules over $R$, describing the rank or dimension of such objects. They can be described in several equivalent ways, either for rectangular matrices, or finitely presented left modules, or pairs of left modules $\cM_1\subseteq \cM_2$, or maps between finitely generated projective left modules, or maps between left modules \cite{Malcolmson80, Schofield, Li19}. They are useful for instance in the study of Kaplansky's direct finiteness conjecture \cite{AOP} and $L^2$-Betti numbers \cite{JZ17, JZ19}, and have attracted much attention recently \cite{AC18, AC20, Elek16}.

Crucial questions about Sylvester rank functions are, given unital rings $R\subseteq S$, when a Sylvester matrix rank function $\rk$ for $R$ can be extended to a Sylvester matrix rank function for $S$ and when such an extension is unique. Of course one needs some conditions on the ring extension $S\supseteq R$. In the literature these questions have been discussed in several different situations.

For a field $\Kb$, Gromov and Elek introduced the notion of amenable $\Kb$-algebras \cite{Gromov99, Elek03a}. This includes the algebras of subexponential growth studied by Rowen \cite{Rowen90}, and is further studied in \cite{ALLWb, Bartholdi08, CSSV, Elek06, Elek17}. In these works people have tried to define dimension for finitely generated modules over an amenable $\Kb$-algebra $S$ using a F{\o}lner sequence of $S$, but it is not clear whether one obtains  a limit along the F{\o}lner sequence or not, as asked by Gromov \cite[page 348]{Gromov99}.

For a twisted crossed product $R*\Gamma$ constructed out of a twisted action of a discrete amenable group $\Gamma$ on $R$  preserving $\rk$, Ara, O'Meara, and Perera showed that $\rk$ can always be extended to a Sylvester matrix rank function $\rk_{R*\Gamma}$ for $R*\Gamma$ using an ultrafilter, and used it to prove Kaplansky's direct finiteness conjecture for free-by-amenable groups \cite{AOP}. They didn't address the question whether $\rk_{R*\Gamma}$ depends on the choice of the ultrafilter or not, equivalently, whether one obtains a limit along a F{\o}lner sequence of $\Gamma$ or not.

When $\Eb/\Kb$ is a field extension and $R$ is a $\Kb$-algebra, Jaikin-Zapirain showed that under some conditions $\rk$ can be extended a Sylvester matrix rank function of $\Eb\otimes_\Kb R$ \cite{JZ19}. More precisely, he constructed a natural extension of $\rk$ to $\Eb\otimes_\Kb R$ when $\Eb/\Kb$ is algebraic, and when $\Eb=\Kb(t)$ for some $t$ transcendental over $\Kb$ and $\rk$ is regular in the sense that it is induced from a Sylvester matrix rank function of a unital von Neumann regular ring $R'$ and a unital ring homomorphism $R\rightarrow R'$. These natural extensions are fundamental in his work on Atiyah conjecture and L\"{u}ck's approximation conjecture. He asked the question whether there is a general construction of a natural extension of $\rk$ to $\Eb\otimes_\Kb R$ unifying his construction in the above two cases \cite[Question 8.8]{JZ17}.

In this work we give a general framework to unify the above situations. We introduce a notion of right amenable extension $S$ of $R$ (see Definition~\ref{D-amenable}). This means that we are given a collection $\cF$ of finitely generated free left $R$-submodules of $S$ satisfying suitable conditions. Intuitively, elements of $\cF$ play the role of right F{\o}lner sets. Then the construction of Ara, O'Meara, and Perera also works in this general situation. Namely, given a non-principal ultrafilter $\omega$ on the set $\hat{\cF}$ of nonzero elements in $\cF$, defining $\rk_{\cF}(A)$ as $\lim_{\cW\rightarrow \omega}\frac{\rk_\cW(A)}{\dim (\cW)}$ for each $A\in M_{n, m}(S)$ yields a Sylvester matrix rank function for $S$, where $\dim(\cW)$ is the rank of $\cW$ as a free left $R$-module and $\rk_\cW(A)$ is $\rk(B)$ for some matrix $B$ over $R$ associated to $\cW$ and $A$.

In order to guarantee that $\rk_{\cF}$ does not depend on the choice of the ultrafilter $\omega$, we introduce a stronger notion of right amenable normal extension $S$ of $R$ preserving $\rk$ (see Definition~\ref{D-normal}). This includes all the above cases except that for the amenable $\Kb$-algebras $S$ studied in \cite{Gromov99, Elek03a} one needs to assume that $S$ has no zero-divisors. Our first main result is the following.

\begin{theorem} \label{T-main infimum for matrix}
Let $\rk$ be a Sylvester matrix rank function of a unital ring $R$, and let $S$ be a right amenable normal extension  of $R$ with a finite approximate system $\cF$ preserving $\rk$.
For each $A\in M_{n, m}(S)$,  $\frac{\rk_\cW(A)}{\dim(\cW)}$ converges to $\inf_{\cV\in \hat{\cF}}\frac{\rk_\cV(A)}{\dim(\cV)}$ when $\cW\in \hat{\cF}$ becomes more and more right invariant.
\end{theorem}

A byproduct of our proof for Theorem~\ref{T-main infimum for matrix} is an affirmative answer to Gromov's question in the case the amenable $\Kb$-algebra $S$ has no zero-divisors (see Corollary~\ref{C-Gromov} and  Remark~\ref{R-Gromov}). In the literature there are two well-known ways to  guarantee the convergence along F{\o}lner sequences of amenable groups. The first is the strong subadditivity, implying an infimum rule, i.e. the limit is the infimum. The second is the subadditivity implying the existence of the limit via the Ornstein-Weiss lemma. In order to prove Theorem~\ref{T-main infimum for matrix}, we establish a linear infimum rule in Lemma~\ref{L-infimum rule}.

The space of all Sylvester matrix rank functions for $R$ is naturally a compact convex space. One can ask whether $\rk_{\cF}$ depends on $\rk$ continuously or not. We have not been able to answer this question in full generality. At the technical level, this requires uniform convergence along some F{\o}lner sequence.
For this purpose we introduce a weak quasitiling property (Definition~\ref{D-weak quasitiling}), which is a weak analogue of the quasitiling of Ornstein and Weiss for amenable groups in \cite{OW}. Twisted crossed products $R*\Gamma$ for amenable groups $\Gamma$ and $\Eb\otimes_\Kb R$ for field extension $\Eb/\Kb$ and $\Kb$-algebras $R$ have the  weak quasitiling property. Our second main result is the following.

\begin{theorem} \label{T-main continuity}
Assume that $S$ is a right  amenable normal extension of $R$ with $(\cF, \cU)$ satisfying the weak quasitiling property. The map from the space of Sylvester matrix rank functions of $R$ preserved by $S$ to the space of Sylvester matrix rank functions of $S$ sending $\rk$ to $\rk_{\cF}$ is affine and continuous.
\end{theorem}

When a  unital $C^*$-algebra $\sA$ has a tracial state $\tr$, it is well known that $\tr$ induces a Sylvester matrix rank function $\rk_\tr$ for $\sA$. In fact, the von Neumann rank function on group algebras of a discrete group $\Gamma$ is constructed this way from the canonical tracial state of the reduced group $C^*$-algebra $C^*_{\rm red}(\Gamma)$ and plays a fundamental role in the study of $L^2$-Betti numbers \cite{JZ19, Luck02}. For a twisted action $(\alpha, u)$ of a discrete group $\Gamma$ on $\sA$, one has the maximal twisted crossed product $C^*$-algebra $\sA\rtimes_{\alpha, u}\Gamma$ \cite{BS, PR89, PR90, PR92}, which contains an algebraic twisted crossed product $\sA*\Gamma$. If the twisted action $(\alpha, u)$ preserves $\tr$, then $\tr$ extends to a tracial state $\tr_{\alpha, u}$ of $\sA\rtimes_{\alpha, u}\Gamma$  naturally, which in turn induces a Sylvester matrix rank function $\rk_{\tr_{\alpha, u}}$ for $\sA\rtimes_{\alpha, u}\Gamma$. If furthermore $\Gamma$ is amenable, then we also have our Sylvester matrix rank function $(\rk_\tr)_{\cF}$ on $\sA*\Gamma$ constructed out of $\rk_\tr$ and the finite approximation system $\cF$ in Example~\ref{E-crossed product}.  In this case we obtain two Sylvester matrix rank functions on $\sA*\Gamma$, namely $(\rk_\tr)_{\cF}$ and the  restriction of  $\rk_{\tr_{\alpha, u}}$ to $\sA*\Gamma$. Our third main result says that these two rank functions coincide.

\begin{theorem} \label{T-same}
Let $(\alpha, u)$ be a twisted action of an amenable discrete group $\Gamma$ on a unital $C^*$-algebra $\sA$ preserving a tracial state $\tr$, and let $\sA\rtimes_{\alpha, u}\Gamma$ be the maximal twisted crossed product $C^*$-algebra. Let $\cF$ be the finite approximation system in Example~\ref{E-crossed product} for $\sA*\Gamma$.
Then
$\rk_{\tr_{\alpha, u}}(A)=(\rk_\tr)_{\cF}(A)$ for every $A\in M_{n, m}(\sA*\Gamma)$.
\end{theorem}

This article is organized as follows. In Section~\ref{S-preliminary} we set up general notation and recall some basic concepts. We introduce the notation of right amenable extensions and discuss some basic properties in Section~\ref{S-amenable extension}. The construction of Ara, O'Meara, and Perera is applied in Section~\ref{S-rank for amenable extension} to right amenable extensions to yield Sylvester matrix rank functions for the extensions. We introduce the notion of right amenable normal extensions in Section~\ref{S-normal amenable extension}. Theorems~\ref{T-main infimum for matrix} and \ref{T-main continuity} are proven in Sections~\ref{S-infimum rule} and \ref{S-continuity} respectively. In Section~\ref{S-bivariant extension} we describe the bivariant Sylvester module rank function corresponding to $\rk_{\cF}$ under the condition of weak quasitiling property.  In Section~\ref{S-field extension} we discuss in detail the case of $\Eb\otimes_\Kb R$ for field extensions $\Eb/\Kb$, showing that $\rk_{\cF}$ behaves well with respect to compositions of field extensions and that our construction extends that of Jaikin-Zapirain, thus answering his question \cite[Question 8.8]{JZ17} affirmatively. Theorem~\ref{T-same} is proven in Section~\ref{S-trace}.
\medskip

\noindent{\it Acknowledgments.}
The second-named author was partially supported by NSF grants DMS-1600717 and DMS-1900746.
Preliminary stages of this work were carried out during his stay in Spring 2018 at ICMAT. He thanks Pere Ara and Andrei Jaikin-Zapirain
for inspiring discussions. We are grateful to the referee for very helpful comments, especially for providing a proof of Proposition 9.6.

\section{Preliminaries} \label{S-preliminary}

Throughout this paper, for a unital ring $R$ (a group $\Gamma$ resp.) we denote by $1_R$ ($e_\Gamma$ resp.) the identity element of $R$ ($\Gamma$ resp.).
For an abelian group $V$ and $n\in \Nb$, we shall write $V^n$ ($V^{n\times 1}$ resp.) for the space of row (column resp.) vectors of length $n$ with entries in $V$.

Let $R$ be a unital ring. A nonzero $a\in R$ is a {\it zero-divisor} if $ab=0$ or $ba=0$ for some nonzero $b\in R$. We say $R$ is a {\it domain} if it has no zero-divisors. Given a left $R$-module $\cM$, for sets $A\subseteq R$ and $\cV\subseteq \cM$, we denote by $A\cV$ the set of finite sums of elements of the form $av$ for $a\in A$ and $v\in \cV$.

\subsection{Sylvester rank functions} \label{SS-Sylvester}

In this subsection we recall the definitions and facts about Sylvester rank functions. We refer the reader to \cite{JZ19, Malcolmson80, Schofield, Li19} for detail.
Let $R$ be a unital ring.
The following definition was given by Malcolmson \cite{Malcolmson80}.

\begin{definition} \label{D-matrix rank}
A {\it Sylvester matrix rank function} for $R$ is an $\Rb_{\ge 0}$-valued function $\rk$ on the set of rectangular matrices over $R$ satisfying the following conditions:
\begin{enumerate}
\item $\rk(0)=0$ and $\rk(1_R)=1$.
\item $\rk(AB)\le \min(\rk(A), \rk(B))$.
\item $\rk(\left[\begin{matrix} A & \\ & B \end{matrix}\right])=\rk(A)+\rk(B)$.
\item $\rk(\left[\begin{matrix} A & C\\ & B \end{matrix}\right])\ge \rk(A)+\rk(B)$.
\end{enumerate}
\end{definition}

In particular, if $A\in M_{n, n}(R)$ and $C\in M_{m, m}(R)$ are invertible, then
$$\rk(ABC)=\rk(B)$$
for every $B\in M_{n, m}(R)$. It follows that for any $A\in M_{n, m}(R)$, $C\in M_{l, m}(R)$ and $B\in M_{l, k}(R)$, one has
$$\rk(\left[\begin{matrix} A & \\ C & B \end{matrix}\right])=\rk(\left[\begin{matrix}  & A\\ B & C \end{matrix}\right])=\rk(\left[\begin{matrix} B & C\\ & A\end{matrix}\right])\ge \rk(A)+\rk(B).$$
For any $A, B\in M_{n, m}(R)$, one also has
$$ \rk(A+B)=\rk(\left[\begin{matrix} I_n & I_n \end{matrix}\right] \left[\begin{matrix} A & \\  & B \end{matrix}\right] \left[\begin{matrix} I_m \\ I_m \end{matrix}\right])\le \rk(\left[\begin{matrix} A & \\ & B \end{matrix}\right])=\rk(A)+\rk(B).$$

Equipped with the pointwise operations, the space $\Pb(R)$ of all Sylvester matrix rank functions for $R$ is a compact convex set. The following definition was given in \cite{Li19}.

\begin{definition} \label{D-bivariant}
A {\it bivariant Sylvester module rank function} for $R$ is an $\Rb_{\ge 0}\cup \{+\infty\}$-valued function $(\cM_1, \cM_2)\mapsto \dim(\cM_1|\cM_2)$ on the class of all pairs of left $R$-modules $\cM_1\subseteq \cM_2$ satisfying the following conditions:
\begin{enumerate}
\item $\dim(\cM_1|\cM_2)$ is an isomorphism invariant.
\item Writing $\dim(\cM)=\dim(\cM|\cM)$, we have $\dim(\{0\})=0$ and $\dim(R)=1$.
\item For any left $R$-modules $\cM_1\subseteq \cM_2$ and $\cM_1'\subseteq \cM_2'$, one has
$$ \dim(\cM_1\oplus \cM_1'|\cM_2\oplus \cM_2')=\dim(\cM_1|\cM_2)+\dim(\cM_1'|\cM_2').$$
\item $\dim(\cM_1|\cM_2)=\sup_{\cM_1'}\dim(\cM_1'|\cM_2)$ for $\cM_1'$ ranging over all finitely generated $R$-submodules of $\cM_1$.
\item When $\cM_1$ is finitely generated, $\dim(\cM_1|\cM_2)=\inf_{\cM_2'}\dim(\cM_1|\cM_2')$ for $\cM_2'$ ranging over all finitely generated $R$-submodules of $\cM_2$ containing $\cM_1$.
\item $\dim(\cM_2)=\dim(\cM_1|\cM_2)+\dim(\cM_2/\cM_1)$.
\end{enumerate}
\end{definition}

There is a natural one-to-one correspondence between Sylvester matrix rank functions for $R$ and bivariant Sylvester module rank functions for $R$ \cite[Theorems 2.4 and 3.3]{Li19}.
Given a bivariant Sylvester module rank function $\dim(\cdot |\cdot)$ for $R$, the corresponding Sylvester matrix rank function $\rk$ is determined by
$$\rk(A)=\dim(R^nA|R^m)$$
for all $A\in M_{n, m}(R)$.

We summarize the properties about bivariant Sylvester module rank  functions which we shall need \cite[Theorem 3.4, Proposition 3.19, Proposition 3.20, Proposition 5.1]{Li19}.

\begin{theorem} \label{T-bivariant}
Let $\dim(\cdot|\cdot)$ be a bivariant Sylvester module rank function for $R$. The following are true.
\begin{enumerate}
\item For any left $R$-modules $\cM_1\subseteq \cM_2\subseteq \cM_3$, one has
$$ \dim(\cM_2|\cM_3)=\dim(\cM_1|\cM_3)+\dim(\cM_2/\cM_1|\cM_3/\cM_1).$$
\item For any left $R$-modules $\cM_1, \cM_2\subseteq \cM_3$, one has
$$ \dim(\cM_1+\cM_2|\cM_3)+\dim(\cM_1\cap \cM_2|\cM_3)\le \dim(\cM_1|\cM_3)+\dim(\cM_2|\cM_3).$$
\item For any left $R$-modules $\cM_1\subseteq \cM_2$ and any $R$-module homomorphism $\pi: \cM_2\rightarrow \cM$, one has
$$\dim(\pi(\cM_1)|\pi(\cM_2))\le \dim(\cM_1|\cM_2).$$
\item For any left $R$-modules $\cM, \cM_1\subseteq \cM_2$ with $\cM_1$ finitely generated, denoting by $\pi_{\cM'}$ the homomorphism $\cM_2\rightarrow \cM_2/\cM'$ for every $R$-submodule $\cM'$ of $\cM_2$, one has
    $$ \dim(\pi_\cM(\cM_1)|\cM_2/\cM)=\inf_{\cM'}\dim(\pi_{\cM'}(\cM_1)|\cM_2/\cM')$$
    for $\cM'$ ranging over all finitely generated $R$-submodules of $\cM$.
\end{enumerate}
\end{theorem}

\subsection{Group rings and twisted crossed products} \label{SS-crossed product}

In this subsection we recall the definitions of group rings and  twisted crossed products. We refer the reader to \cite{Passman, Passman89} for detail.

Let $R$ be a unital ring and $\Gamma$ a discrete group.

The {\it group ring} of $\Gamma$ with coefficients in $R$, denoted by $R\Gamma$, consists of all finitely supported functions $f: \Gamma\rightarrow R$. We shall write $f$ as $\sum_{s\in \Gamma} f_s s$ with $f_s\in R$ for all $s\in \Gamma$ and $f_s=0$ except for finitely many $s\in \Gamma$. The addition and multiplication in $R\Gamma$ are given by
$$ \sum_{s\in \Gamma} f_s s+\sum_{s\in \Gamma} g_s s=\sum_{s\in \Gamma}(f_s+g_s)s, \mbox{ and } (\sum_{s\in \Gamma}f_s s)(\sum_{t\in \Gamma} g_t t)=\sum_{s, t\in \Gamma}f_sg_tst.$$

Given an action $\alpha$ of $\Gamma$ on $R$ via automorphisms, one can define the {\it crossed product} $R\rtimes_\alpha \Gamma$.  It also consists of finitely supported functions $\Gamma\rightarrow R$. The addition and multiplication in $R\rtimes_\alpha \Gamma$ are given by
$$ \sum_{s\in \Gamma} f_s s+\sum_{s\in \Gamma} g_s s=\sum_{s\in \Gamma}(f_s+g_s)s, \mbox{ and } (\sum_{s\in \Gamma}f_s s)(\sum_{t\in \Gamma} g_t t)=\sum_{s, t\in \Gamma}f_s\alpha_s(g_t)st.$$
The group ring $R\Gamma$ is the crossed product $R\rtimes_\alpha \Gamma$ for the trivial action of $\Gamma$ on $R$.

When $\Gamma$ has an action $\beta$ on a discrete group $G$ via automorphisms, one has the semidirect product group $G\rtimes_\beta \Gamma$ defined. As a set, it is $G\times \Gamma$. The multiplication of $G\rtimes_\beta \Gamma$ is given by $(g, s)(h, t)=(g\beta_s(h), st)$. Then the group ring $R (G\rtimes_\beta \Gamma)$ is naturally isomorphic to the crossed product $(RG)\rtimes_\alpha \Gamma$, where $\alpha_s(\sum_{g\in G}f_g g)=\sum_{g\in G} f_g \beta_s(g)$ for $s\in \Gamma$ and $\sum_{g\in G}f_g g\in RG$.

A {\it twisted crossed product}\footnote{This is called {\it crossed product} in \cite{Passman89}. We shall call it twisted crossed product since the corresponding notion in $C^*$-algebras is called so.}  $R*\Gamma$  is a unital $\Gamma$-graded ring, i.e. $R*\Gamma=\bigoplus_{s\in \Gamma}V_s$ with $V_sV_t\subseteq V_{st}$ for all $s, t\in \Gamma$, such that $V_{e_\Gamma}$ is isomorphic to $R$ as a unital ring and $V_s$ contains a unit $\bar{s}$ of $R*\Gamma$ for each $s\in \Gamma$. We shall identify $V_{e_\Gamma}$ with $R$. It  is easily checked that for any $s\in \Gamma$, if some $a\in V_s$ is invertible in $R*\Gamma$, then its inverse lies in $V_{s^{-1}}$. It follows that $\bar{s}$ generates $V_s$ as a free left (right resp.) $R$-module.
Then $\bar{s}$ is unique up to multiplication by some unit in $R$ from either left or right.
Clearly crossed products are twisted crossed products. Given a normal subgroup $G$ of $\Gamma$, it is easily checked that the group ring $R\Gamma$ can be written as $(RG)*(\Gamma/G)$.

\subsection{Amenable groups and amenable $\Kb$-algebras} \label{SS-amenable}

Let $\Gamma$ be a discrete group.
For a nonempty finite set $K\subseteq \Gamma$ and  $\varepsilon>0$, we say a nonempty finite set $F\subseteq \Gamma$ is $(K, \varepsilon)$-invariant if $|FK|<(1+\varepsilon)|F|$. The group $\Gamma$ is {\it amenable} if
for any nonempty finite set $K\subseteq \Gamma$ and any $\varepsilon>0$ there is some $(K, \varepsilon)$-invariant nonempty finite set $F\subseteq \Gamma$.

Let $\Kb$ be a field. The notion of amenable $\Kb$-algebras was introduced by Gromov \cite[Section 1.11]{Gromov99} and Elek \cite{Elek03a}.  Let $S$ be a unital $\Kb$-algebra, i.e. $\Kb$ lies in the center of $S$ and $1_{\Kb}=1_S$. We say that $S$ is {\it right amenable} if for every finite subset $V$ of $S$ and every $\varepsilon>0$ there is some finite-dimensional $\Kb$-linear subspace $\cW$ of $S$ containing $1_S$ such that $\dim_{\Kb}(\cW+\cW V)< (1+\varepsilon)\dim_{\Kb}(\cW)$ (see \cite{ALLWb} for a discussion of the difference between requiring $1_S\in \cW$ or not). When $S$ is a domain, one can drop the condition that $1_S\in \cW$ \cite[Corollary 3.10]{ALLWb}. Finitely generated $\Kb$-algebras of subexponential growth and commutative $\Kb$-algebras are right amenable \cite{Elek03a}. For a discrete group $\Gamma$, the group ring $\Kb \Gamma$ is right amenable if and only if $\Gamma$ is amenable \cite{Elek03a, Elek06, Bartholdi08}.

\subsection{Tracial states and Sylvester matrix rank functions} \label{SS-tracial state}

It is well known that a tracial state on a unital $C^*$-algebra $\sA$ induces a Sylvester matrix rank function for $\sA$. For example, for any discrete group $\Gamma$, the von Neumann Sylvester matrix rank function for the reduced group $C^*$-algebras $C^*_{\rm red}(\Gamma)$ is induced from the canonical tracial state of $C^*_{\rm red}(\Gamma)$ \cite{Luck02, JZ19}. Here we recall the construction. We refer the reader to \cite{KR} for basics on $C^*$-algebras and von Neumann algebras.

A {\it $C^*$-algebra} is a $*$-algebra $\sA$ over $\Cb$ equipped with a complete norm such that $\|ab\|\le \|a\|\cdot \|b\|$ and $\|a^*a\|=\|a\|^2$ for all $a, b\in \sA$. An element $a\in \sA$ is {\it positive} if $a=b^*b$ for some $b\in \sA$. A {\it state} for a unital $C^*$-algebra $\sA$ is a unital linear functional $\varphi: \sA\rightarrow \Cb$ such that it is {\it positive} in the sense that it sends positive elements to positive elements.
A state $\varphi$ of $\sA$ is {\it tracial} if $\varphi(ab)=\varphi(ba)$ for all $a, b\in \sA$. A {\it $*$-representation} of $\sA$ on a Hilbert space $H$ is a $*$-homomorphism from $\sA$ to the $C^*$-algebra $\cB(H)$ of all bounded linear operators from $H$ to itself.

Let $\sA$ be a unital $C^*$-algebra and let $\varphi$ be a state of $\sA$. Then one has the GNS representation $\pi_\varphi$ associated to $\varphi$ as follows. The set $I_\varphi:=\{a\in \sA: \varphi(a^*a)=0\}$ is a left ideal of $\sA$. We have an inner product on $\sA/I_\varphi$ defined by $\left<a+I_\varphi, b+I_\varphi\right>:=\varphi(b^*a)$ for $a, b\in \sA$. Denote by $L^2(\sA, \varphi)$ the completion of $\sA/I_\varphi$ under the norm induced by this inner product. Then $\sA$ has a $*$-representation $\pi_\varphi$ on $L^2(\sA, \varphi)$ determined by $\pi_\varphi(a)(b+I_\varphi)=ab+I_\varphi$ for all $a, b\in \sA$. Denote by $\sA''_\varphi$ the von Neumann algebra generated by $\pi_\varphi(\sA)$, i.e. the closure of $\pi_\varphi(\sA)$ in $\cB(L^2(\sA, \varphi))$ under the weak operator topology. Put $\xi_\varphi=1_\sA+I_\varphi\in L^2(\sA, \varphi)$.
Denote by $\varphi''$ the state on $\sA''_\varphi$ given by
$$\varphi''(T)=\left< T\xi_\varphi, \xi_\varphi\right>.$$
Note that $\varphi''$ extends $\varphi$ in the sense that $\varphi(a)=\varphi''(\pi_\varphi(a))$ for all $a\in \sA$.

Now assume that $\tr$ is a tracial state of $\sA$. It is easily checked
that $\tr''$ is a tracial state on $\sA''_\tr$. We extend $\tr''$ to square matrices over $\sA''_\tr$ by
$$\tr''(A)=\sum_{j=1}^n\tr''(A_{jj})$$
for all $A\in M_n(\sA''_\tr)$. Then $\tr''(AB)=\tr''(BA)$ for all $A\in M_{n, m}(\sA''_\tr)$ and $B\in M_{m, n}(\sA''_\tr)$, and $\tr''(A^*A)\ge 0$ for all $A\in M_{n, m}(\sA''_\tr)$.

Let $A\in M_{n, m}(\sA)$. It follows from von Neumann's double commutant theorem that the orthogonal projection $P_{\overline{\im \pi_\tr(A)}}$ from $L^2(\sA, \tr)^{n\times 1}$ to $\overline{\im \pi_\tr(A)}=\overline{\pi_\tr(A) \cdot L^2(\sA, \tr)^{m\times 1}}$  lies in $M_n(\sA''_\tr)$. Thus we may define
\begin{align} \label{E-def of rank from trace}
\rk_\tr(A):=\tr''(P_{\overline{\im \pi_\tr(A)}})\ge 0.
\end{align}
Using the polar decomposition of $A$ we find some $T\in M_{n, m}(\sA''_\tr)$ such that $T^*T=P_{\overline{\im \pi_\tr(A^*)}}$ and $TT^*=P_{\overline{\im \pi_\tr(A)}}$. Thus
\begin{align} \label{E-rank for adjoint}
\rk_\tr(A)=\rk_\tr(A^*).
\end{align}
Denote by $\ker \pi_\tr(A)$ the set of $x\in L^2(\sA, \tr)^{m\times 1}$ satisfying $\pi_\tr(A)x=0$, and by $P_{\ker \pi_\tr(A)}$ the orthogonal projection from $L^2(\sA, \tr)^{m\times 1}$ to $\ker \pi_\tr(A)$. Then $P_{\ker \pi_\tr(A)}$  lies in $M_m(\sA''_\tr)$, and $P_{\ker \pi_\tr(A)}+P_{\overline{\im \pi_\tr(A^*)}}=I_m$.
Thus
\begin{align} \label{E-rank for kernel}
\rk_\tr(A)=\rk_\tr(A^*)=\tr''(I_m-P_{\ker \pi_\tr(A)})=m-\tr''(P_{\ker \pi_\tr(A)}).
\end{align}

For convenience of the reader, we give a proof of the following proposition.

\begin{proposition} \label{P-trace to rank}
Let $\tr$ be a tracial state of a unital $C^*$-algebra $\sA$. Then $\rk_\tr$ defined via \eqref{E-def of rank from trace} is a Sylvester matrix rank function for $\sA$.
\end{proposition}
\begin{proof}
Clearly $\rk_\tr(0)=0$ and $\rk_\tr(1_\sA)=1$. This verifies the condition (i) in Definition~\ref{D-matrix rank}.

Let $A\in M_{n, m}(\sA)$ and $B\in M_{m, k}(\sA)$. Then $\overline{\im \pi_\tr(AB)}\subseteq \overline{\im \pi_\tr(A)}$. Thus $P_{\overline{\im \pi_\tr(AB)}}\le  P_{\overline{\im \pi_\tr(A)}}$, and hence $\rk_\tr(AB)\le \rk_\tr(A)$. We also have
$$ \rk_\tr(AB)\overset{\eqref{E-rank for adjoint}}=\rk_\tr(B^*A^*)\le \rk_\tr(B^*)\overset{\eqref{E-rank for adjoint}}=\rk_\tr(B).$$
This verifies the condition (ii) in Definition~\ref{D-matrix rank}.

For $A\in M_{n, m}(\sA)$ and $B\in M_{k, l}(\sA)$, putting $D= \left[\begin{matrix} A & \\ & B \end{matrix}\right]$, we have
$P_{\overline{\im \pi_\tr(D)}}=P_{\overline{\im \pi_\tr(A)}}+P_{\overline{\im \pi_\tr(B)}}$, and hence
$$\rk_\tr( \left[\begin{matrix} A & \\ & B \end{matrix}\right])=\rk_\tr(A)+\rk_\tr(B).$$
This verifies the condition (iii) in Definition~\ref{D-matrix rank}.

Let $A\in M_{n, m}(\sA), B\in M_{n, k}(\sA)$, and $C\in M_{l, k}(\sA)$. Put $D=\left[\begin{matrix} A & B\\ & C \end{matrix}\right]$. We have $\overline{\im \pi_\tr(A)}\subseteq \overline{\im \pi_\tr(D)}$. Denote by $H$ the orthogonal complement of $\overline{\im \pi_\tr(A)}$ in $\overline{\im \pi_\tr(D)}$, and by $P_H$ the orthogonal projection from $L^2(\sA, \tr)^{(n+l)\times 1}$ to $H$. Then $P_{\overline{\im \pi_\tr(D)}}=P_{\overline{\im \pi_\tr(A)}}+P_H$. Denote by $Q$ the natural projection map
$L^2(\sA, \tr)^{(n+l)\times 1}=L^2(\sA, \tr)^{n\times 1}\oplus L^2(\sA, \tr)^{l\times 1}\rightarrow L^2(\sA, \tr)^{l\times 1}$. Then $Q\in M_{l, n+l}(\sA)$, and  $Q(H)$ is dense in $\overline{\im \pi_\tr(C)}$. The polar decomposition of $QP_H$ gives us a $T\in M_{l, n+l}(\sA''_\tr)$ such that $TT^*=P_{\overline{\im \pi_\tr(C)}}$, and $T^*T$ is an orthogonal projection satisfying $T^*T\le P_H$. Then
$$ \tr''(P_H)\ge \tr''(T^*T)=\tr''(TT^*)=\tr''(P_{\overline{\im \pi_\tr(C)}})=\rk_\tr(C),$$
and hence
$$ \rk_\tr(\left[\begin{matrix} A & B\\ & C \end{matrix}\right])=\tr''(P_{\overline{\im \pi_\tr(D)}})=\tr''(P_{\overline{\im \pi_\tr(A)}})+\tr''(P_H)\ge \rk_\tr(A)+\rk_\tr(C).$$
This verifies the condition (iv) in Definition~\ref{D-matrix rank}.
\end{proof}

\subsection{Twisted crossed product $C^*$-algebras} \label{SS-twisted}

In this subsection we recall some basic facts about twisted crossed product $C^*$-algebras. We refer the reader to \cite{BS, PR89, PR90, PR92} for more information.

Let $\sA$ be a unital $C^*$-algebra.
An element $u\in \sA$ is a {\it unitary} if $u^*u=uu^*=1_\sA$. For each unitary $u\in \sA$, we have the inner automorphism $\Ad(u)$ of $\sA$ sending $a$ to $uau^*$.
Denote by $\Aut(\sA)$ the automorphism group of $\sA$, and by ${\rm U}(\sA)$ the unitary group of $\sA$.

Let $\Gamma$ be a discrete group with identity element $e_\Gamma$. A {\it twisted action} of $\Gamma$ on $\sA$  \cite[Definition 2.1]{BS} is  a pair $(\alpha, u)$ of maps $\alpha: \Gamma\rightarrow \Aut(\sA)$ and $u: \Gamma\times \Gamma\rightarrow {\rm U}(\sA)$ such that
\begin{enumerate}
\item $\alpha_{e_\Gamma}=\id$ and $u_{e_\Gamma, s}=u_{s, e_\Gamma}=1_\sA$ for all $s\in \Gamma$;
\item $\alpha_s\alpha_t=\Ad(u_{s, t})\alpha_{st}$ for all $s, t\in \Gamma$;
\item $\alpha_\gamma(u_{s, t})u_{\gamma, st}=u_{\gamma, s}u_{\gamma s, t}$ for all $\gamma, s, t\in \Gamma$.
\end{enumerate}

Let $(\alpha, u)$ be a twisted action of $\Gamma$ on $\sA$. Define $\sA*\Gamma$ as the space of finitely supported functions $f: \Gamma\rightarrow \sA$. We shall write $f\in \sA*\Gamma$ as $\sum_{s\in \Gamma} f_s \bar{s}$. Then $\sA*\Gamma$ is a $*$-algebra with the algebraic operations defined by
\begin{align*}
\sum_{s\in \Gamma} f_s \bar{s}+\sum_{s\in \Gamma} g_s \bar{s} &=\sum_{s\in\Gamma}(f_s+g_s)\bar{s}, \\
(\sum_{s\in \Gamma} f_s \bar{s})(\sum_{t\in \Gamma} g_t \bar{t})&=\sum_{s, t\in \Gamma}f_s \alpha_s(g_t)u_{s, t}\overline{st},\\
(\sum_{s\in \Gamma}f_s \bar{s})^*&=\sum_{s\in \Gamma}u_{s^{-1}, s}^*\alpha_{s^{-1}}(f_s^*)\overline{s^{-1}}.
\end{align*}
Clearly $\sA*\Gamma$ is a twisted crossed product in the sense of Section~\ref{SS-crossed product}.

Define a seminorm on $\sA*\Gamma$ by
\begin{align} \label{E-norm}
\big\|f\big\|=\sup_\pi \big\|\pi(f)\big\|
\end{align}
for $\pi$ ranging over all unital $*$-representations of $\sA*\Gamma$ on Hilbert spaces. As we shall see below, this is actually a norm on $\sA*\Gamma$. The {\it maximal twisted crossed product $C^*$-algebra} for $(\alpha, u)$ is the completion of $\sA*\Gamma$ under this norm, and will be denoted  by $\sA\rtimes_{\alpha, u}\Gamma$.

Let $\pi: \sA\rightarrow \cB(H)$ be a unital $*$-representation of $\sA$ on a Hilbert space $H$. Denote by $\ell^2(\Gamma, H)$ the space of all functions $x: \Gamma\rightarrow H$ satisfying $\sum_{t\in \Gamma}\|x_t\|^2<+\infty$. We shall write $x\in \ell^2(\Gamma, H)$ as $\sum_{t\in \Gamma} x_t t$. Then $\ell^2(\Gamma, H)$  is a Hilbert space equipped with the inner product
$$\left<\sum_{t\in \Gamma} x_t t, \sum_{t\in \Gamma} y_t t\right>:=\sum_{t\in \Gamma}\left<x_t, y_t\right>$$
and the associated norm $\big\|\sum_{t\in \Gamma} x_t t\big\|=(\sum_{t\in \Gamma}\|x_t\|^2)^{1/2}$. For each $\sum_{s\in \Gamma} f_s s\in \sA*\Gamma$, define $\pi_{\alpha, u}(\sum_{s\in \Gamma} f_s \bar{s})\in \cB(\ell^2(\Gamma, H))$ by
$$ \pi_{\alpha, u}(\sum_{s\in \Gamma} f_s \bar{s})(\sum_{t\in \Gamma}x_t t)=\sum_{s, t\in \Gamma} \pi(\alpha_{(st)^{-1}}(f_s)u_{(st)^{-1}, s})(x_t)(st).$$
Then $\pi_{\alpha, u}$ is a unital $*$-representation of $\sA*\Gamma$ on $\ell^2(\Gamma, H)$.
When $\pi$ is injective (such $\pi$ always exists), $\pi_{\alpha, u}$ is also injective and hence \eqref{E-norm} does define a norm. In general, $\pi_{\alpha, u}$ extends to a unital $*$-representation of $\sA\rtimes_{\alpha, u}\Gamma$ on $\ell^2(\Gamma, H)$, which we still denote by $\pi_{\alpha, u}$.

Take an injective unital $*$-representation $\pi$ of $\sA$ on some Hilbert space $H$. Denote by $T$ the isometric embedding $H\rightarrow \ell^2(\Gamma, H)$ sending $x$ to $x e_\Gamma$. Then $T^*\pi_{\alpha, u}( \sA\rtimes_{\alpha, u}\Gamma)T=\pi(\sA)$. For each $a\in \sA\rtimes_{\alpha, u}\Gamma$, denote by $\cE(a)$ the unique element in $\sA$ satisfying $T^*\pi_{\alpha, u}(a)T=\pi(\cE(a))$. Then $\cE$ is a unital positive linear map from $\sA\rtimes_{\alpha, u}\Gamma$ to $\sA$ satisfying $\cE(a \overline{e_\Gamma})=a$ for all $a\in \sA$
and $\cE(a \bar{s})=0$ for all $a\in \sA$ and $s\in \Gamma\setminus \{e_\Gamma\}$. In particular, $\cE$ does not depend on the choice of $\pi$.

\section{Amenable extensions} \label{S-amenable extension}

In this section we define amenable extensions and give some basic examples.

A unital ring $R$ is said to have {\it invariant basis number} (IBN) or {\it unique rank property} (URP) if for any distinct $n, m\in \Nb$, the left $R$-modules $R^n$ and $R^m$ are not isomorphic, or equivalently, the right $R$-modules $R^n$ and $R^m$ are not isomorphic \cite[page 3]{Lam}. When $R$ has IBN, for any finitely generated free left $R$-module $\cM$ we write $\dim(\cM)$ for the nonnegative integer satisfying $\cM\cong R^{\dim(\cM)}$.
Note that if $R$ has a Sylvester matrix rank function, then it has IBN.

Let $R$ be a unital ring and $S$ be a unital ring containing $R$ with $1_R=1_S$.

\begin{definition} \label{D-amenable}
Assume that $R$ has IBN.
A collection $\cF$ of left $R$-submodules of $S$ is called a {\it finite approximation system} if the following conditions hold:
\begin{enumerate}
\item Each $\cW\in \cF$ is a finitely generated free left $R$-module.
\item For any $\cW, \cV\in \cF$, one has $\cW\cap \cV, \cW+\cV\in \cF$.
\item For any $\cW, \cV\in \cF$ satisfying $\cV\subseteq \cW$, one has $\cW=\cV\oplus \cV'$ for some $\cV'\in \cF$. In particular, $\dim(\cV)\le \dim(\cW)$.
\item Every finitely generated left $R$-submodule of $S$ is contained in some $\cW\in \cF$.
\end{enumerate}
Denote by $\hat{\cF}$ the set of nonzero elements in $\cF$.
We say that $S$ is  a {\it  right amenable extension} of $R$ with finite approximation system $\cF$ if furthermore
\begin{enumerate}
\item[(v)] For any finite subset $V$ of $S$ and $\varepsilon>0$ there is some $\cW\in \hat{\cF}$ which is $(V, \varepsilon)$-invariant in the sense that one has $\dim(\tilde{\cW})\le (1+\varepsilon)\dim(\cW)$ for some $\tilde{\cW}\in \cF$ containing $\cW+\cW V$.
\end{enumerate}
\end{definition}

\begin{example} \label{E-group ring}
Let $R$ be a field $\Kb$, $\Gamma$ a discrete group, and $S$ the group ring $\Kb \Gamma$.
For each finite set $F\subseteq \Gamma$ denote by $\Kb F$ the set of elements in $\Kb \Gamma$ with support contained in $F$.
Denote by $\cF$ the set  of $\Kb F$ for $F$ ranging over finite subsets of $\Gamma$. Then $\cF$ is a finite approximation system for $\Kb \Gamma$. Clearly $\Kb\Gamma$ is a right amenable extension of $\Kb$ with this $\cF$ if and only if $\Gamma$ is amenable.
\end{example}

\begin{example} \label{E-crossed product}
Let $\Gamma$ be a discrete group and let $R*\Gamma=\bigoplus_{s\in \Gamma}V_s$  be a twisted crossed product as in Section~\ref{SS-crossed product}.
Assume that $R$ has IBN. Let $\cF$ be the set of $R$-modules of the form $\sum_{s\in F}V_s$ where $F$ is a finite subset of $\Gamma$. Then
$\cF$ is a finite approximation system for $R*\Gamma$.
It is easily checked that  the ring $R*\Gamma$ is a right amenable extension of $R$ with the above $\cF$ if and only if $\Gamma$ is amenable.
\end{example}

\begin{example} \label{E-k alg}
Let $\Kb$ be a field and $S$ a unital $\Kb$-algebra.
Let $\cF$ be the set of all finite-dimensional $\Kb$-linear subspaces of $S$. Then $\cF$ is a finite approximation system for $S$.
If $S$ is right amenable over $\Kb$ in the sense of Gromov and Elek as in Section~\ref{SS-amenable}, then clearly $S$ is a right amenable extension of $\Kb$ with the above $\cF$. When $S$ is a domain, the converse also holds.
\end{example}

\begin{example} \label{E-field extension}
Let $\Kb$ be a field and $\Eb$ a field containing $\Kb$. Let $R$ be a $\Kb$-algebra with IBN.
Denote by $\cF$ the set of $V\otimes_\Kb R$ for $V$ ranging over finite-dimensional $\Kb$-linear subspaces of $\Eb$. Then $\cF$ is a finite approximation system for $\Eb\otimes_{\Kb} R$. Since $\Eb$ is right amenable over $\Kb$, $\Eb\otimes_{\Kb} R$ is a right amenable extension of $R$ with this  $\cF$.
\end{example}

We shall use the following elementary lemma a few times.

\begin{lemma} \label{L-additivity for standard}
Let $\cF$ be a finite approximation system. The following are true.
\begin{enumerate}
\item For any $\cV, \cW\in \cF$, we have
$$\dim(\cV+\cW)+\dim(\cV\cap \cW)=\dim(\cV)+\dim(\cW).$$
\item For any $\cW_1, \dots, \cW_n, \cW\in \cF$ with $\cW_1, \dots, \cW_n\subseteq \cW$, one has
$$\dim(\cW)-\dim\big(\bigcap_{j=1}^n\cW_j\big)\le \sum_{j=1}^n(\dim(\cW)-\dim(\cW_j)).$$
\end{enumerate}
\end{lemma}
\begin{proof} (i). We have $\cW=\cW'\oplus (\cV\cap \cW)$ for some $\cW'\in \cF$ and $\cV=\cV'\oplus (\cV\cap \cW)$ for some $\cV'\in \cF$.
Clearly $\cV+\cW=\cV'\oplus (\cV\cap \cW)\oplus \cW'$. Thus
\begin{align*}
 \dim(\cV+\cW)+\dim(\cV\cap \cW)&=\dim(\cV')+\dim(\cV\cap \cW)+\dim(\cW')+\dim(\cV\cap \cW)\\
 &=\dim(\cV)+\dim(\cW).
 \end{align*}

 (ii). We argue by induction on $n$. This is trivial when $n=1$. Suppose that it holds for $n$. Let $\cW_1, \dots, \cW_n, \cW_{n+1}, \cW\in \cF$ with $\cW_1, \dots, \cW_n, \cW_{n+1}\subseteq \cW$. Then
 \begin{align*}
 &\dim(\cW)-\dim\big(\bigcap_{j=1}^{n+1}\cW_j\big)\\
 &=\dim(\cW)-\dim\big(\bigcap_{j=1}^n\cW_j\big)-\dim(\cW_{n+1})+\dim\big(\cW_{n+1}+\bigcap_{j=1}^n\cW_j\big)\\
 &\le \dim(\cW)-\dim\big(\bigcap_{j=1}^n\cW_j\big)-\dim(\cW_{n+1})+\dim(\cW)\\
 &\le \sum_{j=1}^{n+1}(\dim(\cW)-\dim(\cW_j)),
 \end{align*}
 where in the equality we apply part (i) and in the last inequality we apply the inductive hypothesis.
\end{proof}

\section{Sylvester matrix rank function for amenable extensions} \label{S-rank for amenable extension}

In this section we give the Ara-O'Meara-Perera construction of Sylvester matrix rank functions for amenable extensions using ultralimits in \cite{AOP}.
They did the construction for twisted crossed products in Example~\ref{E-crossed product}, but their method works for general case easily.

Let $R$ be a unital ring. Let $\rk$ be a Sylvester matrix rank function for $R$, and $\dim(\cdot|\cdot)$ the corresponding bivariant Sylvester module rank function for $R$. Let $S$ be a unital ring containing $R$ with $1_R=1_S$, and let $\cF$ be a finite approximation system.

Let $A\in M_{n, m}(S)$ and $\cW\in \hat{\cF}$. By the condition (iv) of Definition~\ref{D-amenable} there is some $\tilde{\cW}\in \hat{\cF}$ such that $\cW A_{i, j}\subseteq \tilde{\cW}$ for all $1\le i\le n$ and $1\le j\le m$. Then we have $\cW^nA\subseteq \tilde{\cW}^m$. Take an $R$-basis $w_1, \dots, w_l$ for $\cW^n$, and an $R$-basis $\tilde{w}_1, \dots, \tilde{w}_p$ for $\tilde{\cW}^m$. Then we have
\begin{eqnarray} \label{E-matrix}
\left[\begin{matrix} w_1A \\ \vdots \\ w_lA \end{matrix}\right]=B\left[\begin{matrix} \tilde{w}_1\\ \vdots\\ \tilde{w}_p \end{matrix}\right]
 \end{eqnarray}
for some $B\in M_{l, p}(R)$. Note that $\rk(B)$ does not depend on the choices of the bases $w_1, \dots, w_l$ and $\tilde{w}_1, \dots, \tilde{w}_p$.
In fact, by Lemma~\ref{L-matrix vs module} below $\rk(B)$ does not depend on the choice of $\tilde{\cW}$ either.
Thus we may define
$$\rk_{\cW}(A)=\rk(B).$$

\begin{lemma} \label{L-matrix vs module}
Let $A\in M_{n, m}(S)$ and $\cW, \tilde{\cW}\in \hat{\cF}$ such that $\cW A_{i, j}\subseteq \tilde{\cW}$ for all $1\le i\le n$ and $1\le j\le m$.
If we take an $R$-basis $w_1, \dots, w_l$ for $\cW^n$ and an $R$-basis $\tilde{w}_1, \dots, \tilde{w}_p$ for $\tilde{\cW}^m$, and define $B\in M_{l, p}(R)$ via \eqref{E-matrix}, then
$$ \rk(B)=\dim(\cW^nA|S^m).$$
\end{lemma}
\begin{proof} Clearly $\cW^nA$ is a finitely generated left $R$-submodule of $\tilde{\cW}^m$, and
$$\rk(B)=\dim(\cW^n A|\tilde{\cW}^m).$$
For any $\cV\in \hat{\cF}$ satisfying $\cW^n A\subseteq \cV^m$, by the conditions (ii) and (iii) of Definition~\ref{D-amenable} we know that both $\cV^m$ and $\tilde{\cW}^m$ are direct summands of $(\tilde{\cW}+\cV)^m$ as left $R$-modules, and hence
$$\dim(\cW^n A|\cV^m)=\dim(\cW^n A|(\tilde{\cW}+\cV)^m)=\dim(\cW^n A|\tilde{\cW}^m).$$
By the condition (iv) of Definition~\ref{D-amenable} every finitely generated left $R$-submodule of $S^m$ is contained in $\cV^m$ for some $\cV\in \hat{\cF}$. Therefore
\begin{align*}
\dim(\cW^n A|S^m)&=\inf_\cM\dim(\cW^n A|\cM)\\
&=\inf_{\cV\in \hat{\cF}, \cW^n A\subseteq \cV^m}\dim(\cW^nA|\cV^m)\\
&=\dim(\cW^n A|\tilde{\cW}^m)=\rk(B),
\end{align*}
where in the first line $\cM$ ranges over all finitely generated left $R$-submodules of $S^m$ containing $\cW^n A$.
\end{proof}

We record some basic properties of $\rk_\cW$ in the following obvious lemma which we leave for the reader to check.

\begin{lemma} \label{L-rank on finite}
Let $\cW\in \hat{\cF}$. The following are true.
\begin{enumerate}
\item $\rk_\cW(0)=0$ and $\rk_\cW(1_S)=\dim(\cW)$.
\item For any $A\in M_{n, m}(S)$ and $B\in M_{m, l}(S)$, taking $\tilde{\cW}\in \hat{\cF}$ with $\cW A_{i, j}\subseteq \tilde{\cW}$ for all $1\le i\le n$ and $1\le j\le m$, one has $\rk_\cW(AB)\le \min(\rk_\cW(A), \rk_{\tilde{\cW}}(B))$.
\item $\rk_\cW(\left[\begin{matrix} A & \\ & B \end{matrix}\right])=\rk_\cW(A)+\rk_\cW(B)$.
\item $\rk_\cW(\left[\begin{matrix} A & C\\ & B \end{matrix}\right])\ge \rk_\cW(A)+\rk_\cW(B)$.
\item For any $\cV\in \cF$ with $\cW\subseteq \cV$ and $A\in M_{n, m}(S)$, one has
$$\rk_\cW(A)\le \rk_\cV(A)\le \rk_\cW(A)+n(\dim(\cV)-\dim(\cW)).$$
\end{enumerate}
\end{lemma}

Recall that for a nonempty set $J$, a nonempty family $\omega$ of subsets of $J$ is called a {\it filter} if it is closed under taking finite intersections, $\emptyset\not\in J$,  and for any $X\in \omega$ and $X\subseteq Y\subseteq J$ one has $Y\in\omega$. An {\it ultrafilter} on $J$ is a maximal proper filter $\omega$, i.e. for any proper filter $\omega'$ on $J$ containing $\omega$, one has $\omega=\omega'$. Given any ultrafilter $\omega$ on $J$ and any map $f$ from $J$ to a compact Hausdorff space $Z$, there is a unique $z_0\in Z$ such that for every neighborhood $U$ of $z_0$ in $Z$, the set $f^{-1}(U)$ is in $\omega$. We shall write
$z_0$ as $\lim_{j\to \omega}f(j)$.

Now we assume that $S$ is a right amenable extension of $R$ with $\cF$.
We say that an ultrafilter $\omega$ on $\hat{\cF}$ is {\it non-principal} if for any finite subset $V$ of $S$ and any $\varepsilon>0$, the set of all $(V, \varepsilon)$-invariant $\cW\in \hat{\cF}$ is an element of $\omega$. By Zorn's lemma there are non-principal ultrafilters on $\hat{\cF}$. Fix a non-principal ultrafilter $\omega$ on $\hat{\cF}$.

The following lemma is an immediate consequence of Lemma~\ref{L-rank on finite}.(v).

\begin{lemma} \label{L-boundary}
Let $V$ be a finite subset of $S$. For each $\cW\in \hat{\cF}$ take a $\tilde{\cW}_V\in \hat{\cF}$ with smallest $\dim(\tilde{\cW}_V)$ such that $\cW+\cW V\subseteq \tilde{\cW}_V$. Then
$$\lim_{\cW\rightarrow \omega} \frac{\dim(\tilde{\cW}_V)}{\dim(\cW)}=1 \quad \mbox{ and } \quad \lim_{\cW\rightarrow \omega} \frac{\rk_{\tilde{\cW}_V}(A)-\rk_\cW(A)}{\dim(\cW)}=0 $$
for every $A\in M_{n, m}(S)$.
\end{lemma}

For each $A\in M_{n, m}(S)$, put
$$\rk_{\cF}(A)=\lim_{\cW\rightarrow \omega}\frac{\rk_\cW(A)}{\dim(\cW)}.$$
It follows from Lemmas~\ref{L-rank on finite} and \ref{L-boundary} that $\rk_{\cF}$ is a Sylvester matrix rank function for $S$.

If $R$ has a unique Sylvester matrix rank function $\rk$, then $\rk_{\cF}$ extends $\rk$. In general, in order for $\rk_{\cF}$ to extend $\rk$, we need to assume some extra conditions, which we shall discuss in Section~\ref{S-normal amenable extension}.

\section{Amenable normal extensions} \label{S-normal amenable extension}

In this section we define amenable normal extensions and discuss some basic examples.
Let $R$ be a unital ring with IBN, and let $S$ be a unital ring containing $R$ with $1_S=1_R$.

Recall that a subset $G$ of $S$ is {\it multiplicative} if $1_S\in G$ and $st\in G$ for all $s, t\in G$.
The set of all non-zero-divisors in $S\setminus \{0\}$ is multiplicative.

Denote by $N_S(R)$ the multiplicative set of non-zero-divisors $g$ in $S\setminus \{0\}$ which normalizes $R$, i.e. $gR=Rg$. For each $g\in N_S(R)$ and $a\in R$, there is a unique $\sigma_g(a)\in R$ satisfying $ga=\sigma_g(a)g$. The map $\sigma_g: a\mapsto \sigma_g(a)$ is an automorphism of $R$.

\begin{definition} \label{D-normal}
We say that a finite approximation system $\cF$ for $S$ is {\it normal} if
there is a multiplicative subset $\cU$ of $N_S(R)$ such that:
\begin{enumerate}
\item Each $\cW\in \cF$ has a basis consisting of elements in $\cU$.
\item For each $g\in \cU$, $Rg$ is in $\cF$.
\item For any $g\in \cU$ and $\cW\in \cF$ with $\cW\subseteq Sg$, one has $\{x\in S: xg\in \cW\}\in \cF$.
    \end{enumerate}
   The conditions (i) and (ii) imply that for any $g\in \cU$ and $\cW\in \cF$ one has $g\cW, \cW g\in \cF$. Thus the conditions (i) and (ii) imply (iii) when  $\cU$ is a group.

We call $S$  a {\it right amenable normal extension of $R$} if  it is right amenable with some normal finite approximation system $\cF$, and  given a Sylvester matrix rank function $\rk$ for $R$, call $S$
a {\it  right amenable normal extension of $R$ preserving $\rk$} if furthermore $\rk$ is $\sigma_g$-invariant for every $g\in \cU$.
\end{definition}

\begin{example} \label{E-group ring1}
Let $\Kb, \Gamma$, and $\cF$ be as in Example~\ref{E-group ring}. Then $\cF$ is normal with $\cU=\Gamma$, where we identify $\Gamma$ with its image under the natural embedding $\Gamma\hookrightarrow \Kb\Gamma$. Since $\Kb$ has a unique Sylvester matrix rank function $\rk$, when $\Gamma$ is amenable, $\Kb \Gamma$ is a right amenable normal extension of $\Kb$ preserving $\rk$.
\end{example}

\begin{example} \label{E-crossed product1}
Let $\Gamma, R, R*\Gamma$  and $\cF$ be as in Example~\ref{E-crossed product}. Then $\cF$ is normal with $\cU$ consisting of $b\bar{s}$ for $b$ a unit in $R$ and $s$ an element of $\Gamma$. If $\Gamma$ is amenable, then $R*\Gamma$ is a right amenable normal extension of $R$.
Since every Sylvester matrix rank function is invariant under inner automorphisms, when $\Gamma$ is amenable, if $\rk$ is a Sylvester matrix rank function for $R$ invariant under $\sigma_{\bar{s}}$ for every $s\in \Gamma$, then
$R*\Gamma$ is a right amenable normal extension of $R$ preserving $\rk$.
\end{example}

\begin{example} \label{E-k alg1}
Let $\Kb$ be a field and $S$ a unital $\Kb$-algebra without zero-divisors. Let $\cF$ be as in Example~\ref{E-k alg}. Then $\cF$ is normal with $\cU$ equal to the set of nonzero elements in $S$. Since $\Kb$ has a unique Sylvester matrix rank function $\rk$, if
$S$ is right amenable over $\Kb$ in the sense of Gromov and Elek as in Section~\ref{SS-amenable}, then $S$ is a right amenable normal extension of $\Kb$ preserving $\rk$.
\end{example}

\begin{example} \label{E-field extension1}
Let $\Kb, \Eb, R$  and $\cF$ be as in Example~\ref{E-field extension}.
Then $\cF$  is normal with $\cU=\{a\otimes_{\Kb} 1_R: a\in \Eb, a\neq 0\}$.
For any Sylvester matrix rank function $\rk$ of $R$, $\Eb\otimes_{\Kb} R$ is a right amenable normal extension of $R$ preserving $\rk$.
\end{example}

\begin{remark} \label{R-extension}
If $\rk$ is a Sylvester matrix rank function for $R$ and $S$ is a right amenable normal extension of $R$ preserving $\rk$, then clearly $\rk_{\cF}$ constructed in Section~\ref{S-rank for amenable extension} extends $\rk$.
\end{remark}

For any ring automorphism $\sigma$ of $R$ and  any left $R$-module $\cM$, we denote by $\cM^\sigma$ the left $R$-module $\{\bar{x}: x\in \cM\}$ with addition  $\bar{x}+\bar{y}=\overline{x+y}$ and scalar multiplication $\sigma(a)\bar{x}=\overline{ax}$. The following lemma gives us the implication of invariance of $\rk$ under an automorphism in terms of the corresponding $\dim(\cdot |\cdot)$.

\begin{lemma} \label{L-invariant rk}
Let $\rk$ be a Sylvester matrix rank function for $R$ and $\dim(\cdot|\cdot)$ the corresponding bivariant Sylvester module rank function for $R$.
Let $\sigma$ be a ring automorphism of $R$ preserving $\rk$. Then the following hold.
\begin{enumerate}
\item For any left $R$-modules $\cM_1\subseteq \cM_2$, one has
$$ \dim(\cM_1|\cM_2)=\dim(\cM_1^\sigma|\cM_2^\sigma).$$
\item Let $g\in S$ such that $\sigma(a)g=ga$ for all $a\in R$. For any left $S$-module $\cN$ and any $R$-submodules $\cM_1, \cM_2$ of $\cN$, one has
$$ \dim(\pi_{g\cM_1}(g\cM_2)|\cN/g\cM_1)\le \dim(\pi_{\cM_1}(\cM_2)|\cN/\cM_1),$$
where $\pi_{\cM'}$ denotes the quotient map $\cN\rightarrow \cN/\cM'$ for every $R$-submodule $\cM'$ of $\cN$. In particular,
$$\dim(g\cM_2|\cN)\le \dim(\cM_2|\cN).$$
\end{enumerate}
\end{lemma}
\begin{proof} (i). For any left $R$-modules $\cM_1\subseteq \cM_2$, define
$$ \dim^\sigma(\cM_1|\cM_2)=\dim(\cM_1^\sigma|\cM_2^\sigma).$$
Then clearly $\dim^\sigma(\cdot|\cdot)$ satisfies the conditions in Definition~\ref{D-bivariant}, and thus is a bivariant Sylvester module rank function for $R$. Let $A\in M_{n, m}(R)$. Denote by $e_1, \dots, e_m$ the standard basis of $R^m$. Then $\overline{e_1}, \dots, \overline{e_m}$ is an $R$-basis of $(R^m)^\sigma$. Clearly $(R^nA)^\sigma$ is the $R$-submodule of $(R^m)^\sigma$ generated by $\sum_{j=1}^m\sigma(A_{i, j})\overline{e_j}$ for $i=1, \dots, n$. Therefore
$$\dim^\sigma(R^nA|R^m)=\dim((R^nA)^\sigma|(R^m)^\sigma)=\rk(\sigma(A))=\rk(A).$$
Consequently, $\dim^\sigma(\cdot|\cdot)$ is the bivariant Sylvester module rank function for $R$ corresponding to $\rk$, and hence $\dim^\sigma(\cdot|\cdot)=\dim(\cdot|\cdot)$.

(ii). Denote by $\varphi$ the $R$-module homomorphism $\cN^\sigma\rightarrow \cN$ sending $\bar{x}$ to $gx$. Then $\varphi(\cM_1^\sigma)=g\cM_1$ and $\varphi(\cM_2^\sigma)=g\cM_2$. Thus we get an induced $R$-module homomorphism $\varphi': \cN^\sigma/\cM_1^\sigma\rightarrow \cN/g\cM_1$ sending $\bar{x}+\cM_1^\sigma$ to $gx+g\cM_1$. Clearly $\varphi'((\cM_1^\sigma+\cM_2^\sigma)/\cM_1^\sigma)=(g\cM_1+g\cM_2)/g\cM_1$ and $\varphi'(\cN^\sigma/\cM_1^\sigma)= g\cN/g\cM_1$. Thus
\begin{align*}
\dim(\pi_{g\cM_1}(g\cM_2)|\cN/g\cM_1)&=\dim((g\cM_1+g\cM_2)/g\cM_1|\cN/g\cM_1)\\
&\le \dim((g\cM_1+g\cM_2)/g\cM_1|g\cN/g\cM_1)\\
&= \dim(\varphi'((\cM_1^\sigma+\cM_2^\sigma)/\cM_1^\sigma)|\varphi'(\cN^\sigma/\cM_1^\sigma))\\
&\le \dim((\cM_1^\sigma+\cM_2^\sigma)/\cM_1^\sigma|\cN^\sigma/\cM_1^\sigma)\\
&= \dim(((\cM_1+\cM_2)/\cM_1)^\sigma|(\cN/\cM_1)^\sigma)\\
&= \dim((\cM_1+\cM_2)/\cM_1|\cN/\cM_1)\\
&=\dim(\pi_{\cM_1}(\cM_2)|\cN/\cM_1),
\end{align*}
where in the 2nd inequality we apply Theorem~\ref{T-bivariant}.(iii) and in the 4th equality we apply part (i).
\end{proof}

\section{Linear infimum rule} \label{S-infimum rule}

In this section we prove Theorem~\ref{T-infimum for matrix}, showing that for a right amenable normal extension preserving $\rk$, the Sylvester matrix rank function $\rk_{\cF}$ constructed in Section~\ref{S-rank for amenable extension} is actually a limit and an infimum, so does not depend on the choice of the ultrafilter. For this we establish an infimum rule in Lemma~\ref{L-infimum rule}.

To motivate Lemma~\ref{L-infimum rule}, we recall what happens for amenable groups. Let $\Gamma$ be an amenable group. Denote by $\cF(\Gamma)$ the set of all finite subsets of $\Gamma$, and by $\hat{\cF}(\Gamma)$ the set of all nonempty finite subsets of $\Gamma$. For an $\Rb$-valued function $\psi$ defined on $\cF(\Gamma)$ or $\hat{\cF}(\Gamma)$, we say that {\it $\psi(F)$ converges to $L\in [-\infty, +\infty]$ when $F$ becomes more and more right invariant}, written as $\lim_F\psi(F)=L$, if for any neighborhood $U$ of $L$ in $[-\infty, +\infty]$ there are some $K\in \hat{\cF}(\Gamma)$ and $\delta>0$ such that for any $(K, \delta)$-invariant $F\in \hat{\cF}(\Gamma)$ one has $\psi(F)\in U$.

There are two well-known results yielding the convergence of $\psi(F)$. The first one is the so-called ``infimum rule'', showing that actually $\lim_F\psi(F)=\inf_{F\in \hat{\cF}(\Gamma)}\psi(F)$. Below is one of its forms
\cite[Definitions 2.2.10, 3.1.5, Remark 3.1.7, and Proposition 3.1.9]{Mou85} \cite[Appendix]{Danilenko} \cite[Lemma 3.3]{LT} (see also \cite{DFR} \cite[Section 4.7]{KL16}). The second one is the Ornstein-Weiss lemma \cite[Theorem 6.1]{LW} \cite[Theorem 4.38]{KL16}, which is based on the quasitiling machinery of Ornstein and Weiss for amenable groups \cite{OW}.

\begin{lemma} \label{L-classical infimum rule}
Let $\Gamma$ be an amenable group and $\varphi$ an $\Rb$-valued function on $\cF(\Gamma)$  satisfying the following conditions:
\begin{enumerate}
\item $\varphi(\emptyset)=0$,
\item $\varphi(sF)=\varphi(F)$ for every  $F\in \cF(\Gamma)$ and $s\in \Gamma$,
\item $\varphi(F_1\cup F_2)+\varphi(F_1\cap F_2)\le \varphi(F_1)+\varphi(F_2)$ for all $F_1, F_2\in \cF(\Gamma)$.
\end{enumerate}
Then
$$ \lim_F\frac{\varphi(F)}{|F|}=\inf_{F\in \hat{\cF}(\Gamma)}\frac{\varphi(F)}{|F|}.$$
\end{lemma}

Now let $R$ be a unital ring with IBN and let $S$ be a right amenable normal extension of $R$ with $\cF$ and $\cU$.
Recall that for any $\cW, \cV\in \cF$, $\cW\cV$ denotes the set of finite sums of elements of the form $wv$ for $w\in \cW$ and $v\in \cV$, and hence $\cW\cV\in \cF$.
Here is the linear version of the above definition of limit.

\begin{definition} \label{D-limit}
For $\cV\in \cF$ and $\delta>0$ we say $\cW\in \hat{\cF}$ is {\it $(\cV, \delta)$-invariant} if $\dim(\cW+\cW \cV)\le (1+\delta)\dim(\cW)$.
For an $\Rb$-valued function $\psi$ defined on $\cF$ or $\hat{\cF}$, we say {\it $\psi(\cW)$ converges to $L\in [-\infty, +\infty]$ when $\cW$ becomes more and more right invariant} if for any neighborhood $U$ of $L$ in $[-\infty, +\infty]$ there are some $\cV\in\cF$ and $\delta>0$ such that for any  $(\cV, \delta)$-invariant $\cW\in \hat{\cF}$
we have
$\psi(\cW)\in U$. In such case we write $\lim_{\cW} \psi(\cW)=L$.
\end{definition}

Now we give a linear infimum rule.

\begin{lemma} \label{L-infimum rule}
Let $R$ be a unital ring with IBN and let $S$ be a right amenable normal extension of $R$ with $\cF$ and $\cU$.
Let $\varphi$ be an $\Rb$-valued function on $\cF$  satisfying the following conditions:
\begin{enumerate}
\item $\varphi(\{0\})=0$,
\item $\varphi(g\cV)\le \varphi(\cV)$ for every  $\cV\in \cF$ and $g\in \cU$,
\item $\varphi(\cV+\cW)+\varphi(\cV\cap \cW)\le \varphi(\cV)+\varphi(\cW)$ for all $\cV, \cW\in \cF$.
\end{enumerate}
Then
$$ \lim_{\cW}\frac{\varphi(\cW)}{\dim(\cW)}=\inf_{\cV\in \hat{\cF}}\frac{\varphi(\cV)}{\dim(\cV)}.$$
\end{lemma}
\begin{proof}
Put $L=\inf_{\cV\in \hat{\cF}}\frac{\varphi(\cV)}{\dim(\cV)}\in [-\infty, +\infty)$. Let $r\in \Rb$ with $r>L$. Take $\varepsilon>0$ with $r-2\varepsilon>L$.
The set of $\cV\in \hat{\cF}$ satisfying $\frac{\varphi(\cV)}{\dim(\cV)}<r-2\varepsilon$ is nonempty. Take a $\cV$ in this set with smallest $\dim(\cV)$, and put $C:=\frac{\varphi(\cV)}{\dim(\cV)}<r-2\varepsilon$. Then $\frac{\varphi(\cV')}{\dim(\cV')}\ge C$ for every $\cV'\in \hat{\cF}$ satisfying $\dim(\cV')<\dim(\cV)$. From the condition (i) we conclude that
$$\varphi(\cV')\ge C \dim(\cV')$$
for every  $\cV'\in \cF$ satisfying $\dim(\cV')<\dim(\cV)$.

We claim that $\varphi(\cW\cV)\le C\dim(\cW\cV)$ for every $\cW\in \cF$.
We prove it by induction on $\dim(\cW)$. The case $\dim(\cW)=0$ follows from the condition (i).
When $\dim(\cW)=1$, we have $\cW=Rg$ for some $g\in \cU$ and hence
$$\varphi(\cW \cV)=\varphi(g \cV)\le \varphi(\cV)=C\dim(\cV)=C\dim(\cW\cV),$$
where the inequality comes from the condition (ii).
Suppose that the claim holds when $\dim(\cW)=n$. Let $\cW\in \cF$ with $\dim(\cW)=n+1$. Write $\cW$ as $\cW_1\oplus \cW_2$ with $\cW_1, \cW_2\in \cF$ such that $\dim(\cW_1)=n$ and $\dim(\cW_2)=1$.
Then $\cW_2=Rg$ for some $g\in \cU$.
By induction hypothesis we have $\varphi(\cW_j\cV)\le C\dim(\cW_j\cV)$ for $j=1, 2$.
If $\cW_1\cV\cap \cW_2\cV=\cW_2\cV$, then $\cW\cV=\cW_1\cV+\cW_2\cV=\cW_1\cV$, and hence
\begin{align*}
\varphi(\cW\cV)=\varphi(\cW_1\cV)\le C\dim(\cW_1\cV_1)=C\dim(\cW\cV).
\end{align*}
Thus we may assume that $\cW_1\cV\cap \cW_2\cV\neq \cW_2\cV$. Note that $\cW_1\cV\cap \cW_2\cV\in \cF$.
Thus
$$\dim(\cW_1\cV\cap \cW_2\cV)<\dim(\cW_2\cV)=\dim(\cV).$$
Therefore $\varphi(\cW_1\cV\cap \cW_2\cV)\ge C\dim(\cW_1\cV\cap \cW_2\cV)$.
By the condition (iii)  we have
\begin{align*}
\varphi(\cW\cV)&=\varphi(\cW_1\cV+\cW_2\cV)\\
&\le \varphi(\cW_1\cV)+\varphi(\cW_2\cV)-\varphi(\cW_1\cV\cap \cW_2\cV)\\
&\le C\dim(\cW_1\cV)+C\dim(\cW_2\cV)-C\dim(\cW_1\cV\cap \cW_2\cV)\\
&=C\dim(\cW_1\cV+\cW_2\cV)=C\dim(\cW\cV),
\end{align*}
where the 2nd equality comes from Lemma~\ref{L-additivity for standard}.(i).
This proves our claim.

From the condition (ii) we have $\varphi(Rg)=\varphi(gR)\le \varphi(R)$ for every $g\in \cU$.
Then from the conditions (i) and (iii)
we have
$\varphi(\cW)\le \varphi(R)\dim(\cW)$ for all $\cW\in \cF$.

Take $\delta>0$ with $\delta \max(|C|, |\varphi(R)|)\dim(\cV)\le \varepsilon$.
Now let $\cW\in \hat{\cF}$ be $(\cV, \delta)$-invariant.
Write $\cV$ as $\bigoplus_{1\le j\le \dim(\cV)}Rv_j$ for some $v_1, \dots, v_{\dim(\cV)}\in \cU$.
For each $1\le j\le \dim(\cV)$,
set $\cW_j=\{w\in \cW: wv_j\in \cW\}=\{x\in S: xv_j\in \cW\cap \cW v_j\}$. Then $\cW_j\in \cF$, and $w\mapsto wv_j$ is a left $R$-module isomorphism from $\cW_j$ onto $\cW\cap \cW v_j$. Note that $\cW+\cW v_j\in \cF$ is contained in  $\cW+\cW\cV$, thus $\dim(\cW+\cW v_j)\le \dim(\cW+\cW\cV)$.
Therefore
\begin{align*}
\dim(\cW)-\dim(\cW_j)&=\dim(\cW)-\dim(\cW\cap \cW v_j)\\
&=\dim(\cW+\cW v_j)-\dim(\cW v_j)\\
&\le \dim(\cW+\cW\cV)-\dim(\cW)\\
&\le \delta \dim(\cW),
\end{align*}
where in the 2nd equality we apply Lemma~\ref{L-additivity for standard}.(i).
Set $\cW^\dag=\bigcap_{j=1}^{\dim(\cV)}\cW_j\in \cF$.
Then $\cW^\dag \cV\subseteq \cW$, and by Lemma~\ref{L-additivity for standard}.(ii) we have
$$\dim(\cW)-\dim(\cW^\dag)\le \sum_{j=1}^{\dim(\cV)}(\dim(\cW)-\dim(\cW_j))\le \delta \dim(\cV)\dim(\cW).$$
Since $\cW^\dag\cV, \cW\in \cF$ and $\cW^\dag \cV\subseteq \cW$, we have $\cW=\cW^\dag\cV\oplus \cW^\sharp$ for some $\cW^\sharp\in \cF$.
Then
$$\dim(\cW)\ge \dim(\cW^\dag \cV)\ge \dim(\cW^\dag v_1)=\dim(\cW^\dag)\ge \dim(\cW)-\delta \dim(\cV)\dim(\cW),$$
and hence
$$ |C\dim(\cW^\dag \cV)-C\dim(\cW)|\le |C|\delta\dim(\cV)\dim(\cW)\le \varepsilon \dim(\cW),$$
and
$$ |\varphi(R)\dim(\cW^\sharp)|=|\varphi(R)| \cdot |\dim(\cW)-\dim(\cW^\dag \cV)|\le |\varphi(R)| \delta \dim(\cV)\dim(\cW)\le \varepsilon \dim(\cW). $$
We conclude that
\begin{align*}
\varphi(\cW)&=\varphi(\cW^\dag\cV+\cW^\sharp)\\
&=\varphi(\cW^\dag\cV+\cW^\sharp)+\varphi(\cW^\dag\cV\cap \cW^\sharp)\\
&\le \varphi(\cW^\dag \cV)+\varphi(\cW^\sharp)\\
&\le C\dim(\cW^\dag \cV)+\varphi(R)\dim(\cW^\sharp)\\
&\le C\dim(\cW)+\varepsilon \dim(\cW)+\varepsilon\dim(\cW)\\
&\le r\dim(\cW),
\end{align*}
where the 1st inequality comes from the condition (iii).
\end{proof}

If we take $R$ to be a field $\Kb$ and $S$ to be the group ring $\Kb \Gamma$ of an amenable group $\Gamma$ in Example~\ref{E-group ring1}, then Lemma~\ref{L-infimum rule} yields a new proof of Lemma~\ref{L-classical infimum rule}.

Now let $\rk$ be a Sylvester matrix rank function for $R$ and $\dim(\cdot|\cdot)$ the corresponding bivariant Sylvester module rank function for $R$. We assume further that $S$ is a right amenable normal extension of $R$ preserving $\rk$.

\begin{lemma} \label{L-amenable extension relative}
Let $\cN_1\subseteq \cN_2$ be left $S$-modules such that $\cN_1$ is finitely generated. Let $\cM_1$ be a finitely generated $R$-submodule of $\cN_1$ generating $\cN_1$ as an $S$-module. Then
\begin{align} \label{E-amenable extension relative}
\lim_\cW\frac{\dim(\cW\cM_1|\cN_2)}{\dim(\cW)}=\inf_{\cV\in \hat{\cF}}\frac{\dim(\cV\cM_1|\cN_2)}{\dim(\cV)}.
\end{align}
This limit does not depend on the choice of $\cM_1$. We denote it by $\dim_{\cF}(\cN_1|\cN_2)$.
\end{lemma}
\begin{proof} Consider the function $\varphi_{\cM_1}$ on $\cF$ defined by $\varphi_{\cM_1}(\cV)=\dim(\cV\cM_1|\cN_2)$. Let's check that it satisfies the conditions in Lemma~\ref{L-infimum rule}. Clearly $\varphi_{\cM_1}(\{0\})=\dim(\{0\}|\cN_2)=0$. For any $g\in \cU$ and $\cV\in \cF$, by Lemma~\ref{L-invariant rk}.(ii) we have
$$ \varphi_{\cM_1}(g\cV)=\dim(g\cV\cM_1|\cN_2)\le \dim(\cV\cM_1|\cN_2)=\varphi_{\cM_1}(\cV).$$
For any $\cV, \cW\in \cF$, by  Theorem~\ref{T-bivariant}.(ii) we have
\begin{align*}
\varphi_{\cM_1}(\cV+\cW)+\varphi_{\cM_1}(\cV\cap \cW)&\le \dim(\cV\cM_1+\cW\cM_1|\cN_2)+\dim(\cV\cM_1\cap \cW\cM_1|\cN_2)\\
&\le \dim(\cV\cM_1|\cN_2)+\dim(\cW\cM_1|\cN_2)\\
&=\varphi_{\cM_1}(\cV)+\varphi_{\cM_1}(\cW).
\end{align*}
Therefore from Lemma~\ref{L-infimum rule} we get \eqref{E-amenable extension relative}.
Denote $\lim_\cW\frac{\dim(\cW\cM_1|\cN_2)}{\dim(\cW)}$ by $L_{\cM_1}$.

Let $\cM_1'$ be another finitely generated $R$-submodule of $\cN_1$ generating $\cN_1$ as an $S$-module. We shall show $L_{\cM_1}=L_{\cM_1'}$. By symmetry it suffices to show $L_{\cM_1'}\le L_{\cM_1}$. Since $\cM_1$ generates $\cN_1$ as an $S$-module, we have $\cM_1'\subseteq \cV'\cM_1$ for some nonzero finitely generated left $R$-submodule $\cV'$ of $S$. We may assume that $\cV'\in \cF$.
For any $\delta>0$ and $\cV\in \cF$, when $\cW$ in $\hat{\cF}(V)$ is $(\cV'+\cV'\cV, \delta)$-invariant,
we have
$$\dim(\cW\cV'+\cW\cV'\cV)=\dim(\cW(\cV'+\cV'\cV))\le (1+\delta)\dim(\cW)\le (1+\delta)\dim(\cW\cV').$$
This shows that when $\cW\in \hat{\cF}$ becomes more and more right invariant, so does $\cW\cV'$. Also clearly $\lim_\cW\frac{\dim(\cW\cV')}{\dim(\cW)}=1$. Therefore
\begin{align*}
L_{\cM_1'}&\le L_{\cV'\cM_1}=\lim_\cW\frac{\dim(\cW\cV'\cM_1|\cN_2)}{\dim(\cW)}
=\lim_\cW\frac{\dim(\cW\cV'\cM_1|\cN_2)}{\dim(\cW\cV')}
=L_{\cM_1}.
\end{align*}
\end{proof}

When $R$ is a field $\Kb$, it has a unique Sylvester matrix  rank function $\rk$, and the corresponding bivariant Sylvester module rank function is $\dim(\cM_1|\cM_2)=\dim_\Kb(\cM_1)$ for all left $R$-modules $\cM_1\subseteq \cM_2$, where $\dim_{\Kb}(\cM_1)$ is the usual dimension for $\Kb$-vector spaces and we put $\dim_{\Kb}(\cM_1)=\infty$ whenever $\cM_1$ is infinite-dimensional. Taking $S$ to be a unital right amenable $\Kb$-algebra without zero-divisors in Example~\ref{E-k alg1}, we obtain the following consequence of Lemma~\ref{L-amenable extension relative}.

\begin{corollary} \label{C-Gromov}
Let $\Kb$ be a field and let $S$ be a unital right amenable $\Kb$-algebra without zero-divisors. Put $\cF$ to be the set of all finite-dimensional $\Kb$-linear subspaces of $S$. For any finitely generated left $S$-module $\cN$ and any finite-dimensional $\Kb$-linear subspace $\cM$ of $\cN$ generating $\cN$ as an $S$-module, the limit $\lim_\cW \frac{\dim_{\Kb}(\cW \cM)}{\dim_{\Kb}(\cW)}$ exists.
\end{corollary}

\begin{remark} \label{R-Gromov}
In \cite[page 348]{Gromov99} Gromov asked whether the limit $\lim_\cW \frac{\dim_{\Kb}(\cW \cM)}{\dim_{\Kb}(\cW)}$ exists for every unital right amenable $\Kb$-algebra $S$.
In \cite[page 477]{Elek03a} Elek constructed an example showing that in general the limit does not exist. In this example, $S$ is the unital $\Kb$-algebra generated by $x$ and $y$ subject to $x^2=0$ and $xy=0$. The left $S$-module $\cN$ is the $S$-submodule of $S$ generated by $x$ and $y$, and one can take $\cM=\Kb x+\Kb y$. Corollary~\ref{C-Gromov} answers Gromov's question affirmatively in the case $S$ is a domain.
\end{remark}

Theorem~\ref{T-main infimum for matrix} is part of the following result.

\begin{theorem} \label{T-infimum for matrix}
Let $A\in M_{n, m}(S)$. Then
$$ \lim_\cW\frac{\rk_\cW(A)}{\dim(\cW)}=\inf_{\cV\in \hat{\cF}}\frac{\rk_\cV(A)}{\dim(\cV)}=\dim_{\cF}(S^nA|S^m).$$
\end{theorem}
\begin{proof} Denote by $\cM_1$ the left $R$-submodule of $S^m$ generated by the rows of $A$.
For each $\cW\in \hat{\cF}$, clearly $\cW^nA=\cW\cM_1$, and hence by Lemma~\ref{L-matrix vs module} we have
$$\rk_\cW(A)=\dim(\cW^nA|S^m)=\dim(\cW\cM_1|S^m).$$
Put $\cN_1=S^nA$ and $\cN_2=S^m$. Now the theorem follows from Lemma~\ref{L-amenable extension relative}.
\end{proof}

\section{Continuity of extension} \label{S-continuity}

In this section we prove Theorem~\ref{T-continuity}, establishing the continuity of the map $\rk\mapsto \rk_{\cF}$. For this purpose we need to assume some linear quasitiling property. To motivate this property, we recall first the Ornstein-Weiss quasitiling theorem for amenable groups \cite{OW} \cite[Theorem 4.36]{KL16}.

\begin{theorem} \label{T-OW quasitiling}
Let $\Gamma$ be an amenable group. Let $0<\varepsilon<1$ and $K\in \hat{\cF}(\Gamma)$. Then there are some $n\in \Nb$, $(K, \varepsilon)$-invariant  $F_1, \dots, F_n\in \hat{\cF}(\Gamma)$, $\delta>0$ and $K'\in \hat{\cF}(\Gamma)$ such that every $(K', \delta)$-invariant  $F\in \hat{\cF}(\Gamma)$ can be $\varepsilon$-quasitiled by $F_1, \dots, F_n$ in the sense that there are $C_1, \dots, C_n\subseteq F$ and $F_{j, c}\subseteq G$ for each $1\le j\le n$ and $c\in C_j$ satisfying the following conditions:
\begin{enumerate}
\item For each $1\le j\le n$ and $c\in C_j$, one has $F_{j, c}\subseteq F_j$ and $|F_{j, c}|\ge (1-\varepsilon)|F_j|$.
\item The sets $cF_{j, c}$ for $1\le j\le n$ and $c\in C_j$ are pairwise disjoint.
\item $\bigcup_{1\le j\le n}\bigcup_{c\in C_j}c F_j\subseteq F$ and $\big|\bigcup_{1\le j\le n}\bigcup_{c\in C_j}c F_j\big|\ge (1-\varepsilon)|F|$.
\end{enumerate}
\end{theorem}

Theorem~\ref{T-OW quasitiling} says that one can choose the shapes $F_1, \dots, F_n$ to be as right invariant as one wants (i.e. $(K, \varepsilon)$-invariant), so that any sufficiently right invariant $F$ (i.e. $(K', \delta)$-invariant) is almost a disjoint union of  translates of $F_1, \dots, F_n$ (i.e. $cF_j$). When $\Gamma=\Zb$, one can take $n=1$ and $F_1=\{0, 1, \dots, N-1\}$ for sufficiently large $N$, and take $C_1$ to be the set of $c\in N\Zb$  satisfying $c+F_1\subseteq F$.  In general, if $\Gamma$ is {\it monotileable} in the sense that there is a right F{\o}lner sequence $\{K_m\}_{m\in \Nb}$ of $\Gamma$ such that for each $m\in \Nb$ one can write $\Gamma$ as a disjoint union of sets of the form $cF_m$ with $c\in \Gamma$, then one can always take $n=1$ and $F_1=K_m$ for sufficiently large $m$. Every amenable residually finite group is monotileable \cite{Weiss}. It is an open question whether every amenable group is   monotileable  or not.

Let $R$ be a unital ring with IBN, and let $S$ be a right amenable normal extension of $R$ with $\cF$ and $\cU$. Here is our linear version of a weak form of the quasitiling in Theorem~\ref{T-OW quasitiling}.

\begin{definition} \label{D-weak quasitiling}
We say that $(\cF, \cU)$ has the {\it weak quasitiling property} if for any $0<\varepsilon<1$ and $\cV\in \hat{\cF}$ there are some $n\in \Nb$ and $(\cV, \varepsilon)$-invariant $\cW_1, \dots, \cW_n\in \hat{\cF}$ such that for any $\delta>0$ and $\cV'\in \hat{\cF}$, there is a $(\cV', \delta)$-invariant $\cW\in \hat{\cF}$ which can be $\varepsilon$-quasitiled by $\cW_1, \dots, \cW_n$ in the sense that there are $C_1, \dots, C_n\subseteq \cU\cap \cW$ and $\cW_{j, c}\in \cF$ for each $1\le j\le n$ and $c\in C_j$ satisfying the following conditions:
\begin{enumerate}
\item For each $1\le j\le n$ and $c\in C_j$, one has $\cW_{j, c}\subseteq \cW_j$ and $\dim(\cW_{j, c})\ge (1-\varepsilon)\dim(\cW_j)$.
\item The left $R$-modules $c\cW_{j, c}$ for $1\le j\le n$ and $c\in C_j$ are linearly independent.
\item $\sum_{1\le j\le n}\sum_{c\in C_j}c\cW_j\subseteq \cW$ and
$\dim(\sum_{1\le j\le n}\sum_{c\in C_j}c\cW_j)\ge (1-\varepsilon)\dim(\cW)$.
\end{enumerate}
\end{definition}

\begin{example} \label{E-crossed product2}
Let $\Gamma, R, R*\Gamma, \cF$ and $\cU$ be as in Example~\ref{E-crossed product1}. If $\Gamma$ is amenable, then from Theorem~\ref{T-OW quasitiling} we know that $(\cF, \cU)$ has the weak quasitiling property. In fact, Theorem~\ref{T-OW quasitiling} implies that $(\cF, \cU)$  satisfies some property stronger than the weak quasitiling property, namely for any $0<\varepsilon<1$ and $\cV\in \hat{\cF}$ there are some $n\in \Nb$, $(\cV, \varepsilon)$-invariant $\cW_1, \dots, \cW_n\in \hat{\cF}$, $\delta>0$ and $\cV'\in \hat{\cF}$ such that every $(\cV', \delta)$-invariant $\cW\in \hat{\cF}$ can be $\varepsilon$-quasitiled by $\cW_1, \dots, \cW_n$.
\end{example}

\begin{lemma} \label{L-field extension amenable}
Let $\Kb$ be a field and $\Eb$ a field containing $\Kb$. Denote by $\hat{\cF}$ the set of nonzero finite-dimensional $\Kb$-linear subspaces of $\Eb$.
For any $\cV\in \hat{\cF}$ and $\varepsilon>0$, there is some $(\cV, \varepsilon)$-invariant $\cW_1\in \hat{\cF}$ such that for any $\cV'\in \hat{\cF}$ and $\delta>0$ there are some
$(\cV', \delta)$-invariant $\cW\in \hat{\cF}$ and a finite set $C_1$ of nonzero elements in $\cW$ with $\cW=\bigoplus_{c\in C_1}c\cW_1$.
\end{lemma}
\begin{proof} If $\Eb$ is algebraic over $\Kb$, then we may take $\cW_1$ to be the subfield of $\Eb$ generated by $\Kb\cup \cV$, $\cW$ to be the subfield of $\Eb$ generated by $\Kb\cup \cV\cup \cV'$, and $C_1$ to be a $\cW_1$-basis of $\cW$. Thus we may assume that $\Eb$ is not algebraic over $\Kb$.

We may assume that $1_{\Kb}\in \cV$.

Let $X$ be a maximal set of elements in $\Eb$ being algebraically independent over $\Kb$. Then $X$ is nonempty. Denote by $\Eb_1$ the subfield of $\Eb$ generated by $X$ and $\Kb$. Then $\Eb$ is algebraic over $\Eb_1$. We may identify $\Eb_1$ with the field of rational functions with coefficients in $\Kb$ and indeterminants in $X$. Denote by $P$ the set of all polynomials with coefficients in $\Kb$ and indeterminants in $X$.

Take a finite field extension $\Fb$ of $\Eb_1$ contained in $\Eb$ such that $\cV\subseteq \Fb$. Take a $\Kb$-basis $A$ for $\cV$, and take an $\Eb_1$-basis $B$ for $\Fb$ with $1_{\Eb}\in B$. Then we can find some nonzero $h\in P$ such that
for any $a\in A$ and $b\in B$ one has $hba=\sum_{b'\in B}f_{a, b, b'}b'$ for some $f_{a, b, b'}\in P$.
Let $Y$ be a finite subset of $X$ containing all the indeterminants appearing in $f_{a, b, b'}$ for all $a\in A$ and $b, b'\in B$.
Denote by $m$ the maximum of the degree of each $y\in Y$ in $f_{a, b, b'}$ for $a$ ranging over elements of $A$ and $b, b'$ ranging over elements of $B$.
For each $n\in \Nb$, denote by $Q_{Y, n}$ the set of polynomials with coefficients in $\Kb$ and indeterminants in $Y$ and degree at most $n-1$ in each $y\in Y$.
Then $\dim_{\Kb}Q_{Y, n}=(1+n)^{|Y|}$.
Take $n\in \Nb$ large enough such that
$\dim_{\Kb}(Q_{Y, n+m})\le (1+\varepsilon)\dim_{\Kb}(Q_{Y, n-1})$.
Note that $f_{a, b, b'}Q_{Y, n-1}\subseteq Q_{Y, n+m}$ for all $a\in A$ and $b, b'\in B$.
Set $\cW_1=\sum_{b\in B}bQ_{Y, n-1}\in \hat{\cF}$. Then
$$ h\cW_1\cV= h\sum_{b\in B}bQ_{Y, n-1}\sum_{a\in A}\Kb a\subseteq\sum_{a\in A}\sum_{b, b'\in B}f_{a, b, b'}b'Q_{Y, n-1}\subseteq \sum_{b'\in B}b'Q_{Y, n+m},$$
and hence
\begin{align*}
\dim_{\Kb}(\cW_1\cV)&=\dim_{\Kb}(h\cW_1\cV)\\
&\le \dim_{\Kb}(\sum_{b\in B}bQ_{Y, n+m})\\
&=|B|\dim_{\Kb}(Q_{Y, n+m})\\
&\le (1+\varepsilon)|B|\dim_{\Kb}(Q_{Y, n-1})\\
&=(1+\varepsilon)\dim_{\Kb}(\cW_1).
\end{align*}
Thus $\cW_1$ is $(\cV, \varepsilon)$-invariant.

Let $\cV'\in \hat{\cF}$  containing $1_{\Kb}$ and $\delta>0$. Then we have $\Fb', A', B', h', Y', m', n', Q_{Y', n'-1}, Q_{Y', n'+m'}$ and $\cW'_1$ as above for $\cV'$ and $\delta$. We may choose $\Fb'\supseteq \Fb$. Take an $\Fb$-basis $Z$ for $\Fb'$ containing $1_{\Eb}$.  We may choose $B'$ to the set of $bz$ for all $b\in B$ and $z\in Z$. We may also choose $Y'\supseteq Y$ and $n'$ to be a multiple of $n$.
Denote by $C_1$ the set of $gz$ for all $z\in Z$ and monomials $g$ in $Y'$ such that in $g$ each $y\in Y$ has degree $\ell_yn$ for some integer $0\le \ell_y<n'/n$ and each $y'\in Y'\setminus Y$ has degree at most $n'-1$. Then
\begin{align*}
\cW'_1=\bigoplus_{b\in B, z\in Z}bzQ_{Y', n'-1}=\bigoplus_{b\in B}b\bigoplus_{z\in Z}zQ_{Y', n'-1}=\bigoplus_{b\in B}b\bigoplus_{c\in C_1}cQ_{Y, n-1}=\bigoplus_{c\in C_1}c\cW_1.
\end{align*}
Since $1_{\Kb}\in B$, we have $C_1\subseteq \cW'_1$.
Putting $\cW=\cW'_1$ finishes the proof.
\end{proof}

\begin{example} \label{E-field extension2}
Let $\Kb, \Eb, R, \cF$ and $\cU$ be as in Example~\ref{E-field extension1}.
It follows from Lemma~\ref{L-field extension amenable} that $(\cF, \cU)$ has the weak quasitiling property.
\end{example}

\begin{question} \label{Q-quasitiling}
Does every right amenable normal extension have the weak quasitiling property?
\end{question}

For any $\cV\in \hat{\cF}$, we say $\cW_1, \cW_2\in \cF$ are {\it $\cV$-separated} if $\cW_1\cV\cap \cW_2\cV=\{0\}$.

\begin{lemma} \label{L-continuous limit}
Let $C>0$ and $\cV\in \hat{\cF}$. Denote by $\Phi_{C, \cV}$ the set of functions $\varphi: \cF\rightarrow \Rb$ satisfying the conditions (i)-(iii) in Lemma~\ref{L-infimum rule} and
\begin{enumerate}
\item[(iv)] For any $\cV$-separated $\cW_1,\cW_2\in \cF$, one has $\varphi(\cW_1+\cW_2)=\varphi(\cW_1)+\varphi(\cW_2)$,
\item[(v)] $\varphi(g\cW)=\varphi(\cW)$ for all $\cW\in \cF$ and $g\in \cU$,
\item[(vi)] $\varphi(R)\le C$,
\item[(vii)] $\varphi(\cW_1)\le \varphi(\cW_2)$ for all $\cW_1, \cW_2\in \cF$ with $\cW_1\subseteq \cW_2$.
\end{enumerate}
Equip $\Phi_{C, \cV}$ with the pointwise convergence topology. Assume that $(\cF, \cU)$ has the weak quasitiling property.
Then the function $\varphi\mapsto \lim_\cW\frac{\varphi(\cW)}{\dim(\cW)}$ on $\Phi_{C, \cV}$ is continuous.
\end{lemma}
\begin{proof} From the conditions (i) and (vii) we know that every $\varphi\in \Phi_{C, \cV}$ is $\Rb_{\ge 0}$-valued. For each $\varphi\in \Phi_{C, \cV}$, put $f(\varphi)=\lim_\cW\frac{\varphi(\cW)}{\dim(\cW)}=\inf_{\cW\in \hat{\cF}}\frac{\varphi(\cW)}{\dim(\cW)}\in \Rb_{\ge 0}$.

Let $\varphi\in \Phi_{C, \cV}$ and $\varepsilon>0$. Take a $\cW^\flat\in \hat{\cF}$ with $\frac{\varphi(\cW^\flat)}{\dim(\cW^\flat)}<f(\varphi)+\varepsilon$. Denote by $U_1$ the set of all $\varphi'\in \Phi_{C, \cV}$ satisfying $\varphi'(\cW^\flat)<\varphi(\cW^\flat)+\varepsilon$. This is an open neighborhood  of $\varphi$ in $\Phi_{C, \cV}$. For each $\varphi'\in U_1$, we have
$$ f(\varphi')\le \frac{\varphi'(\cW^\flat)}{\dim(\cW^\flat)}\le \frac{\varphi(\cW^\flat)}{\dim(\cW^\flat)}+\varepsilon\le f(\varphi)+2\varepsilon.$$

If $f(\varphi)=0$, then $f(\varphi')\ge f(\varphi)$ for all $\varphi'\in U_1$. Thus we may assume $f(\varphi)>0$ and $\varepsilon<f(\varphi)/4$. Take $\delta>0$ such that $2C\delta \dim(\cV)<\varepsilon$. Also take $0<\theta<\delta$ such that $(f(\varphi)-2\varepsilon)(1-\theta)>f(\varphi)-3\varepsilon$.
By assumption we can find some $n\in \Nb$ and $(\cV, \delta)$-invariant $\cW_1, \dots, \cW_n\in \hat{\cF}$ such that for any $\cV_1\in \hat{\cF}$ and $\delta_1>0$ there is some $(\cV_1, \delta_1)$-invariant $\cW\in \hat{\cF}$ which can be $\theta$-quasitiled by $\cW_1, \dots, \cW_n$. Denote by $U_2$ the set of all $\varphi'\in \Phi_{C, \cV}$ satisfying $\varphi'(\cW_j)>\varphi(\cW_j)-\varepsilon$ for all $1\le j\le n$. This is an open neighborhood  of $\varphi$ in $\Phi_{C, \cV}$. Let $\varphi'\in U_2$.
There are some $\cV_1\in \hat{\cF}$ and $\delta_1>0$ such that for any $(\cV_1, \delta_1)$-invariant $\cW\in \hat{\cF}$ one has
$\frac{\varphi'(\cW)}{\dim(\cW)}\le f(\varphi')+\varepsilon$. Take a $(\cV_1, \delta_1)$-invariant $\cW\in \hat{\cF}$ which can be $\theta$-quasitiled by $\cW_1, \dots, \cW_n$. Then we have $C_1, \dots, C_n$ and $\cW_{j, c}$ for $1\le j\le n$ and $c\in C_j$ as in Definition~\ref{D-weak quasitiling}.
Take an $R$-basis $v_1, \dots, v_{\dim(\cV)}$ of $\cV$ in $\cU$. For each $1\le j\le n$, $c\in C_j$, and $1\le i\le \dim(\cV)$, denote by $\cW_{j, c, i}$ the set of $w\in \cW_{j}$ satisfying $wv_i\in \cW_{j, c}\cap \cW_{j} v_i$. Then $\cW_{j, c, i}\in \cF$ and $\dim(\cW_{j, c, i})=\dim(\cW_{j, c}\cap \cW_{j} v_i)$. Note that
\begin{align*}
\dim(\cW_{j} v_i)-\dim(\cW_{j, c}\cap \cW_{j} v_i)&=\dim(\cW_{j, c}+\cW_j v_i)-\dim(\cW_{j, c})\\
&\le \dim(\cW_j+\cW_j\cV)-\dim(\cW_{j, c})\\
&\le (\delta+\theta) \dim(\cW_j)\\
&\le 2\delta \dim(\cW_j),
\end{align*}
where in the equality we apply Lemma~\ref{L-additivity for standard}.(i).
Put $\cW_{j, c}'=\bigcap_{1\le i\le \dim(\cV)}\cW_{j, c, i}\in \cF$. Then
\begin{align*}
 \dim(\cW_{j})-\dim(\cW_{j, c}')&\le \sum_{i=1}^{\dim(\cV)}(\dim(\cW_{j})-\dim(\cW_{j,c, i}))\\
 &=\sum_{i=1}^{\dim(\cV)}(\dim(\cW_{j} v_i)-\dim(\cW_{j, c}\cap \cW_{j} v_i))\\
 &\le 2\delta \dim(\cV)\dim(\cW_{j}),
 \end{align*}
 where in the first inequality we apply Lemma~\ref{L-additivity for standard}.(ii).
 We have $\cW_j=\cW_{j, c}'\oplus \cW_{j, c}''$ for some $\cW_{j, c}''\in \cF$.
 Then
\begin{align*}
\varphi'(\cW_j)&\le \varphi'(\cW_{j, c}')+\varphi'(\cW_{j, c}'')\\
&\le \varphi'(\cW_{j, c}')+\dim(\cW_{j, c}'')\varphi'(R)\\
&= \varphi'(\cW_{j, c}')+(\dim(\cW_{j})-\dim(\cW_{j, c}'))\varphi'(R)\\
&\le \varphi'(\cW_{j, c}')+2C\delta \dim(\cV)\dim(\cW_{j})\\
&\le  \varphi'(\cW_{j, c}')+\varepsilon\dim(\cW_{j}).
\end{align*}
 Note that $c\cW_{j, c}'\cV\subseteq c\cW_{j, c}$. Thus
 \begin{align*}
 \varphi'(\cW)&\ge  \varphi'(\sum_{1\le j\le n}\sum_{c\in C_j}c\cW_{j, c}')\\
 &=\sum_{1\le j\le n}\sum_{c\in C_j}\varphi'(c\cW_{j, c}')\\
 &=\sum_{1\le j\le n}\sum_{c\in C_j}\varphi'(\cW_{j, c}')\\
 &\ge \sum_{1\le j\le n}\sum_{c\in C_j}(\varphi'(\cW_{j})-\varepsilon\dim(\cW_{j}))\\
  &\ge \sum_{1\le j\le n}\sum_{c\in C_j}(\varphi(\cW_{j})-\varepsilon \dim(\cW_j)-\varepsilon\dim(\cW_{j}))\\
  &\ge \sum_{1\le j\le n}\sum_{c\in C_j}(f(\varphi)\dim(\cW_{j})-2\varepsilon\dim(\cW_j))\\
  &=(f(\varphi)-2\varepsilon)\sum_{1\le j\le n}\sum_{c\in C_j}\dim(c\cW_{j})\\
  &\ge (f(\varphi)-2\varepsilon)\dim(\sum_{1\le j\le n}\sum_{c\in C_j}c\cW_j)\\
  &\ge (f(\varphi)-2\varepsilon)(1-\theta)\dim(\cW)\\
  &\ge (f(\varphi)-3\varepsilon)\dim(\cW),
 \end{align*}
 where in the first equality we apply the condition (iv) and in the 5th inequality we apply Lemma~\ref{L-additivity for standard}.(i).
Therefore
$$f(\varphi')\ge \frac{\varphi'(\cW)}{\dim(\cW)}-\varepsilon\ge f(\varphi)-4\varepsilon.$$
It follows that $f$ is continuous on $\Phi_{C, \cV}$.
\end{proof}

Denote by $\Pb_{\cU}(R)$ the set of $\rk\in \Pb(R)$ which is $\sigma_g$-invariant for all $g\in \cU$. This is a compact convex subset of $\Pb(R)$.
A map $f: X\rightarrow Y$ between convex sets is called {\it affine} if $f(tx_1+(1-t)x_2)=tf(x_1)+(1-t)f(x_2)$ for all $0\le t\le 1$ and $x_1, x_2\in X$.
The following is a restatement of Theorem~\ref{T-main continuity}.

\begin{theorem} \label{T-continuity}
Assume that $(\cF, \cU)$ has the weak quasitiling property. The map $\Pb_{\cU}(R)\rightarrow \Pb(S)$ sending $\rk$ to $\rk_{\cF}$ is affine and continuous.
\end{theorem}
\begin{proof} Clearly this map is affine. We just need to show that it is continuous.

Let $A\in M_{n, m}(S)$. Take $\cV\in \hat{\cF}$ with $A_{i, j}\in \cV$ for all $1\le i\le n$ and $1\le j\le m$. Put $C=n$. We have $\Phi_{C, \cV}$ in Lemma~\ref{L-continuous limit}. Now it suffices to show that for each $\rk\in \Pb_{\cU}(R)$, the function $\varphi: \cF\rightarrow \Rb_{\ge 0}$ sending $\cW\in \hat{\cF}$ to $\rk_\cW(A)$ and $\{0\}$ to $0$ lies in $\Phi_{C, \cV}$.

Denote by $\cM_1$ the left $R$-submodule of $S^m$ generated by the rows of $A$. Put $\cN_2=S^m$. It was shown in the proof of Lemma~\ref{L-amenable extension relative} that the function $\cW\mapsto \dim(\cW \cM_1|\cN_2)$ on $\cF$ satisfies the conditions (i)-(iii) in Lemma~\ref{L-infimum rule}. Note that $\cW^n A=\cW \cM_1$ for every $\cW\in \cF$. By Lemma~\ref{L-matrix vs module} we have $\varphi(\cW)=\rk_\cW(A)=\dim(\cW \cM_1|\cN_2)$ for every $\cW\in \hat{\cF}$. Thus $\varphi(\cW)=\dim(\cW \cM_1|\cN_2)$ for every $\cW\in \cF$. Therefore $\varphi$ satisfies the conditions (i)-(iii) in Lemma~\ref{L-infimum rule}.

Let $\cW_1, \cW_2\in \cF$ be $\cV$-separated. If $\cW_1$ or $\cW_2$ is equal to $\{0\}$, say $\cW_1=\{0\}$, then
$$ \varphi(\cW_1+\cW_2)=\varphi(\cW_2)=\varphi(\cW_1)+\varphi(\cW_2).$$
Thus we may assume that $\cW_1, \cW_2\in \hat{\cF}$.
Since $\cV\neq \{0\}$, we have $\cW_1\cap \cW_2=\{0\}$. Take an $R$-basis $w_1, \dots, w_l$ for $\cW_1^n$, an $R$-basis $w_{l+1}, \dots, w_{l+r}$ for $\cW_2^n$, an $R$-basis $\tilde{w}_1, \dots, \tilde{w}_p$ for $(\cW_1\cV)^m$, and an $R$-basis $\tilde{w}_{p+1}, \dots, \tilde{w}_{p+q}$ for $(\cW_2\cV)^m$. Then $w_1, \dots, w_l, w_{l+1}, \dots, w_{l+r}$ is an $R$-basis for $(\cW_1+\cW_2)^n$, and
$\tilde{w}_1, \dots, \tilde{w}_p, \tilde{w}_{p+1}, \dots, \tilde{w}_{p+q}$ is an $R$-basis for $(\cW_1\cV+ \cW_2\cV)^m=((\cW_1+\cW_2)\cV)^m$.
We have
\begin{align*}
\left[\begin{matrix} w_1A \\ \vdots \\ w_lA \end{matrix}\right]=B_1\left[\begin{matrix} \tilde{w}_1\\ \vdots\\ \tilde{w}_p \end{matrix}\right] \mbox{ and } \left[\begin{matrix} w_{l+1}A \\ \vdots \\ w_{l+r}A \end{matrix}\right]=B_2\left[\begin{matrix} \tilde{w}_{p+1}\\ \vdots\\ \tilde{w}_{p+q} \end{matrix}\right]
\end{align*}
for some $B_1\in M_{l, p}(R)$ and $B_2\in M_{r, q}(R)$. Then
\begin{align*}
\left[\begin{matrix} w_1A \\ \vdots \\ w_{l+r}A \end{matrix}\right]=\left[\begin{matrix} B_1 & \\ & B_2 \end{matrix}\right]\left[\begin{matrix} \tilde{w}_1\\ \vdots\\ \tilde{w}_{p+q} \end{matrix}\right].
\end{align*}
Thus
$$\varphi(\cW_1+\cW_2)=\rk(\left[\begin{matrix} B_1 & \\ & B_2 \end{matrix}\right])=\rk(B_1)+\rk(B_2)=\varphi(\cW_1)+\varphi(\cW_2).$$
This verifies the condition (iv) of Lemma~\ref{L-continuous limit}.

Let $g\in \cU$. If $\cW_1=\{0\}$, then $\varphi(g\cW_1)=\varphi(\{0\})=\varphi(\cW_1)$. Thus we may assume $\cW_1\in \hat{\cF}$. Then $gw_1, \dots, gw_l$ is an $R$-basis of $(g\cW_1)^n$, and $g\tilde{w}_1, \dots, g\tilde{w}_p$ is an $R$-basis for $(g\cW_1\cV)^m$. Note that
\begin{align*}
\left[\begin{matrix} gw_1A \\ \vdots \\ gw_l A \end{matrix}\right]=gB_1\left[\begin{matrix} \tilde{w}_1\\ \vdots\\ \tilde{w}_p \end{matrix}\right]=\sigma_g(B_1)\left[\begin{matrix} g\tilde{w}_1\\ \vdots\\ g\tilde{w}_{p} \end{matrix}\right].
\end{align*}
Thus
$$\varphi(g\cW_1)=\rk(\sigma_g(B_1))=\rk(B_1)=\varphi(\cW_1).$$
This verifies the condition (v) of Lemma~\ref{L-continuous limit}.

For any $\cW\in \cF$ we have $\varphi(\cW)\le \dim(\cW^n)=n\dim(\cW)$, verifying the condition (vi) of Lemma~\ref{L-continuous limit}.

For any $\cW_1, \cW_2\in \cF$ with $\cW_1\subseteq \cW_2$, we have
$$\varphi(\cW_1)=\dim(\cW_1\cM_1|\cN_2)\le \dim(\cW_2\cM_1|\cN_2)=\varphi(\cW_2),$$
verifying the condition (vii) of Lemma~\ref{L-continuous limit}.
\end{proof}

\section{Bivariant Sylvester module rank function for amenable normal extensions} \label{S-bivariant extension}

In this section we describe the bivariant Sylvester module rank function corresponding to the Sylvester matrix rank function $\rk_{\cF}$. Throughout this section we assume that $R$ is a unital ring with a Sylvester matrix rank function $\rk$, and $S$ is a right amenable normal extension of $R$ with $\cF$ and $\cU$ preserving $\rk$. Then we have the Sylvester matrix rank function $\rk_{\cF}$ for $S$ extending $\rk$. In Lemma~\ref{L-amenable extension relative} we have defined $\dim_{\cF}(\cN_1|\cN_2)$ for left $S$-modules $\cN_1\subseteq \cN_2$ in the case $\cN_1$ is finitely generated.

Taking $A=(1_S)$ in Theorem~\ref{T-infimum for matrix} we get
$\dim_{\cF}(S|S)=1$.

\begin{lemma} \label{L-amenable extension continuity}
Let $\cN_1\subseteq \cN_2$ be left $S$-modules such that $\cN_1$ is finitely generated. Then
$$\dim_{\cF}(\cN_1|\cN_2)=\inf_{\cN_2^\sharp}\dim_{\cF}(\cN_1|\cN_2^\sharp)$$
for $\cN_2^\sharp$ ranging over finitely generated $S$-submodules of $\cN_2$ containing $\cN_1$.
\end{lemma}
\begin{proof} Clearly $\dim_{\cF}(\cN_1|\cN_2)\le \dim_{\cF}(\cN_1|\cN_2^\sharp)$ for every $S$-submodule $\cN_2^\sharp$ of $\cN_2$ containing $\cN_1$.
Take a finitely generated $R$-submodule $\cM_1$ of $\cN_1$ generating $\cN_1$ as an $S$-module. Let $\varepsilon>0$. Take a  $\cV\in \hat{\cF}$ with
$$\frac{\dim(\cV\cM_1|\cN_2)}{\dim(\cV)}\le \dim_{\cF}(\cN_1|\cN_2)+\varepsilon.$$
Take a finitely generated $R$-submodule $\cM_2$ of $\cN_2$ containing $\cM_1+\cV\cM_1$ with
$$\dim(\cV\cM_1|\cM_2)\le \dim(\cV\cM_1|\cN_2)+\varepsilon.$$
Denote by $\cN_2^\sharp$ the $S$-submodule of $\cN_2$ generated by $\cM_2$. Then $\cN_1\subseteq \cN_2^\sharp$, and
\begin{align*}
\dim_{\cF}(\cN_1|\cN_2^\sharp)&\le \frac{\dim(\cV\cM_1|\cN_2^\sharp)}{\dim(\cV)}\\
&\le \frac{\dim(\cV\cM_1|\cM_2)}{\dim(\cV)}\\
&\le \frac{\dim(\cV\cM_1|\cN_2)}{\dim(\cV)}+\varepsilon\\
&\le \dim_{\cF}(\cN_1|\cN_2)+2\varepsilon.
\end{align*}
\end{proof}

For any left $S$-modules $\cN_1\subseteq \cN_2\subseteq \cN_3$ with $\cN_1, \cN_2$ finitely generated, clearly we have $\dim_{\cF}(\cN_1|\cN_3)\le \dim_{\cF}(\cN_2|\cN_3)$.
For any left $S$-modules $\cN_1\subseteq \cN_2$, we define
$$\dim_{\cF}(\cN_1|\cN_2):=\sup_{\cN_1^\sharp}\dim_{\cF}(\cN_1^\sharp|\cN_2)$$
for $\cN_1^\sharp$ ranging over finitely generated $S$-submodules of $\cN_1$. Then $\dim_{\cF}(\cdot|\cdot)$ clearly satisfies all the conditions in Definition~\ref{D-bivariant}  except the condition (vi).

\begin{lemma} \label{L-amenable extension additivity}
For any left $S$-modules $\cN_1\subseteq \cN_2\subseteq \cN_3$, we have
$$ \dim_{\cF}(\cN_2|\cN_3)\ge \dim_{\cF}(\cN_1|\cN_3)+\dim_{\cF}(\cN_2/\cN_1|\cN_3/\cN_1).$$
\end{lemma}
\begin{proof} Denote by $\pi$ the quotient map $\cN_3\rightarrow \cN_3/\cN_1$. Let $\cM_j$ be a finitely generated $R$-submodule of $\cN_j$ for $j=1, 2$. It suffices to show
\begin{align} \label{E-amenable extension additivity}
\lim_\cW\frac{\dim(\cW(\cM_1+\cM_2)|\cN_3)}{\dim(\cW)}\ge \lim_\cW\frac{\dim(\cW\cM_1|\cN_3)}{\dim(\cW)}+\lim_\cW\frac{\dim(\cW\pi(\cM_2)|\cN_3/\cN_1)}{\dim(\cW)}.
\end{align}
For each $R$-submodule $\cM$ of $\cN_3$ denote by $\pi_\cM$ the quotient map $\cN_3\rightarrow \cN_3/\cM$. For each $\cW\in \cF$ we have
\begin{align*}
\dim(\cW(\cM_1+\cM_2)|\cN_3)&= \dim(\cW\cM_1|\cN_3)+\dim(\pi_{\cW\cM_1}(\cW\cM_2)|\cN_3/\cW\cM_1)\\
&\ge \dim(\cW\cM_1|\cN_3)+\dim(\pi(\cW\cM_2)|\cN_3/\cN_1),
\end{align*}
where the equality comes from Theorem~\ref{T-bivariant}.(i) and the inequality comes from Theorem~\ref{T-bivariant}.(iii) and the fact that $\pi$ factors through $\pi_{\cW\cM_1}$.
Now \eqref{E-amenable extension additivity} follows immediately.
\end{proof}

\begin{theorem} \label{T-amenable extension}
Assume that $(\cF, \cU)$ has the weak quasitiling property.
Then $\dim_{\cF}(\cdot|\cdot)$ is the bivariant Sylvester module rank function for $S$ corresponding to $\rk_{\cF}$.
\end{theorem}
\begin{proof} We shall show that for any left $S$-modules $\cN_1\subseteq \cN_2\subseteq \cN_3$, one has
$$ \dim_{\cF}(\cN_2|\cN_3)=\dim_{\cF}(\cN_1|\cN_3)+\dim_{\cF}(\cN_2/\cN_1|\cN_3/\cN_1).$$
By Lemma~\ref{L-amenable extension additivity} it suffices to show
$$ \dim_{\cF}(\cN_2|\cN_3)\le \dim_{\cF}(\cN_1|\cN_3)+\dim_{\cF}(\cN_2/\cN_1|\cN_3/\cN_1).$$
Denote by $\pi$ the quotient map $\cN_3\rightarrow \cN_3/\cN_1$. Then it suffices to show
\begin{align} \label{E-amenable extension3}
 \dim_{\cF}(\cN_2^\sharp|\cN_3)\le \dim_{\cF}(\cN_1|\cN_3)+\dim_{\cF}(\pi(\cN_2^\sharp)|\cN_3/\cN_1)
\end{align}
for every finitely generated $S$-submodule $\cN_2^\sharp$ of $\cN_2$.

For each $R$-submodule $\cM$ of $\cN_3$ denote by $\pi_\cM$ the quotient map $\cN_3\rightarrow \cN_3/\cM$.
Let $\cM_2$ be a finitely generated $R$-submodule of $\cN_2^\sharp$ generating $\cN_2^\sharp$ as an $S$-module.

Let $\varepsilon>0$. Take $0<\theta<1$ such that
$$\frac{1}{1-\theta}(\dim_{\cF}(\pi(\cN_2^\sharp)|\cN_3/\cN_1)+2\varepsilon)+\theta\dim(\cM_2|\cN_3)<\dim_{\cF}(\pi(\cN_2^\sharp)|\cN_3/\cN_1)+3\varepsilon.$$
Note that
$$ \dim_{\cF}(\pi(\cN_2^\sharp)|\cN_3/\cN_1)=\lim_{\cW}\frac{\dim(\cW\pi(\cM_2)|\cN_3/\cN_1)}{\dim(\cW)}.$$
By assumption we can find some $n\in \Nb$ and
$\cW_1, \dots, \cW_n\in \hat{\cF}$  with
\begin{align} \label{E-amenable extension2}
\frac{\dim(\cW_j\pi(\cM_2)|\cN_3/\cN_1)}{\dim(\cW_j)}\le \dim_{\cF}(\pi(\cN_2^\sharp)|\cN_3/\cN_1)+\varepsilon
\end{align}
for all $1\le j\le n$
such that for any $\cV\in \hat{\cF}$ and $\delta>0$ there is some $(\cV, \delta)$-invariant $\cW\in \hat{\cF}$ which can be $\theta$-quasitiled by $\cW_1, \cdots, \cW_n$.

By parts (iv) and (iii) of Theorem~\ref{T-bivariant} we can find a finitely generated $R$-submodule $\cM_1$ of $\cN_1$ with
\begin{align} \label{E-amenable extension}
\dim(\pi_{\cM_1}(\cW_j\cM_2)|\cN_3/\cM_1)\le \dim(\pi(\cW_j\cM_2)|\cN_3/\cN_1)+\varepsilon
\end{align}
for all $1\le j\le n$.
Denote by $\cN_1^\sharp$ the $S$-submodule of $\cN_1$ generated by $\cM_1$.

By Theorem~\ref{T-bivariant}.(i) we have
\begin{align*}
\lim_{\cW}\frac{\dim(\pi_{\cW\cM_1}(\cW\cM_2)|\cN_3/\cW\cM_1)}{\dim(\cW)}
&=\lim_\cW\frac{\dim(\cW(\cM_1+\cM_2)|\cN_3)}{\dim(\cW)}-\lim_\cW \frac{\dim(\cW\cM_1|\cN_3)}{\dim(\cW)}\\
&=\dim_{\cF}(\cN_1^\sharp+\cN_2^\sharp|\cN_3)-\dim_{\cF}(\cN_1^\sharp|\cN_3).
\end{align*}
Then there are some $\cV\in \hat{\cF}$  and $\delta>0$ such that for every $(\cV, \delta)$-invariant $\cW\in \hat{\cF}$ one has
\begin{align} \label{E-amenable extension1}
 \dim_{\cF}(\cN_1^\sharp+\cN_2^\sharp|\cN_3)-\dim_{\cF}(\cN_1^\sharp|\cN_3)\le \frac{\dim(\pi_{\cW\cM_1}(\cW\cM_2)|\cN_3/\cW\cM_1)}{\dim(\cW)}+\varepsilon.
\end{align}
Take a $(\cV, \delta)$-invariant $\cW\in \hat{\cF}$ such that
$\cW$ can be $\theta$-quasitiled by $\cW_1, \dots, \cW_n$. Then \eqref{E-amenable extension1} holds. Let $C_1, \dots, C_n\subseteq \cU\cap \cW$ and $\cW_{j, c}$ for $1\le j\le n$ and $c\in C_j$ be as in Definition~\ref{D-weak quasitiling}.
Then
\begin{align} \label{E-amenable extension4}
\sum_{1\le j\le n}|C_j|\dim(\cW_j)&\le \frac{1}{1-\theta}\sum_{1\le j\le n}\sum_{c\in C_j}\dim(\cW_{j, c})\\
\nonumber &= \frac{1}{1-\theta}\sum_{1\le j\le n}\sum_{c\in C_j}\dim(c\cW_{j, c})\\
\nonumber &=\frac{1}{1-\theta}\dim(\sum_{1\le j\le n}\sum_{c\in C_j}c\cW_{j, c})\\
\nonumber &\le \frac{1}{1-\theta} \dim(\cW).
\end{align}
We have $\cW=\cW'\oplus \sum_{1\le j\le n}\sum_{c\in C_j}c\cW_j$ for some $\cW'\in \cF$. Take an $R$-basis $\Lambda$ of $\cW'$ in $\cU$. Then
$$|\Lambda|=\dim(\cW')=\dim(\cW)-\dim(\sum_{1\le j\le n}\sum_{c\in C_j}c\cW_j)\le \theta\dim(\cW).$$
Now we have
\begin{align*}
\dim(\pi_{\cW\cM_1}(\sum_{1\le j\le n}\sum_{c\in C_j}c\cW_j\cM_2)|\cN_3/\cW\cM_1)
&\le \sum_{1\le j\le n}\sum_{c\in C_j}\dim(\pi_{\cW\cM_1}(c\cW_j\cM_2)|\cN_3/\cW\cM_1)\\
&\le \sum_{1\le j\le n}\sum_{c\in C_j}\dim(\pi_{c\cM_1}(c\cW_j\cM_2)|\cN_3/c\cM_1)\\
&\le \sum_{1\le j\le n}\sum_{c\in C_j}\dim(\pi_{\cM_1}(\cW_j\cM_2)|\cN_3/\cM_1)\\
&\overset{\eqref{E-amenable extension}}\le \sum_{1\le j\le n}|C_j|(\dim(\pi(\cW_j\cM_2)|\cN_3/\cN_1)+\varepsilon)\\
&\overset{\eqref{E-amenable extension2}}\le \sum_{1\le j\le n}|C_j|(\dim_{\cF}(\pi(\cN_2^\sharp)|\cN_3/\cN_1)+2\varepsilon)\dim(\cW_j)\\
&\overset{\eqref{E-amenable extension4}}\le \frac{1}{1-\theta} \dim(\cW)(\dim_{\cF}(\pi(\cN_2^\sharp)|\cN_3/\cN_1)+2\varepsilon),
\end{align*}
where the first inequality comes from Theorem~\ref{T-bivariant}.(ii), the 2nd inequality comes from Theorem~\ref{T-bivariant}.(iii), and the 3rd inequality comes from Lemma~\ref{L-invariant rk}.(ii),
and hence
\begin{align*}
&\dim(\pi_{\cW\cM_1}(\cW\cM_2)|\cN_3/\cW\cM_1)\\
&\le \dim(\pi_{\cW\cM_1}(\sum_{1\le j\le n}\sum_{c\in C_j}c\cW_j\cM_2)|\cN_3/\cW\cM_1)+\sum_{g\in \Lambda}\dim(\pi_{\cW\cM_1}(g\cM_2)|\cN_3/\cW\cM_1)\\
&\le \frac{1}{1-\theta} \dim(\cW)(\dim_{\cF}(\pi(\cN_2^\sharp)|\cN_3/\cN_1)+2\varepsilon)+\sum_{g\in \Lambda}\dim(g\cM_2|\cN_3)\\
&\le \frac{1}{1-\theta} \dim(\cW)(\dim_{\cF}(\pi(\cN_2^\sharp)|\cN_3/\cN_1)+2\varepsilon)+\sum_{g\in \Lambda}\dim(\cM_2|\cN_3)\\
&\le \frac{1}{1-\theta} \dim(\cW)(\dim_{\cF}(\pi(\cN_2^\sharp)|\cN_3/\cN_1)+2\varepsilon)+\theta \dim(\cW)\dim(\cM_2|\cN_3)\\
&\le \dim(\cW)(\dim_{\cF}(\pi(\cN_2^\sharp)|\cN_3/\cN_1)+3\varepsilon),
\end{align*}
where the first inequality comes from Theorem~\ref{T-bivariant}.(ii), the 2nd inequality comes from Theorem~\ref{T-bivariant}.(iii), and the 3rd inequality comes from Lemma~\ref{L-invariant rk}.(ii).
Thus
\begin{align*}
\dim_{\cF}(\cN_1^\sharp+\cN_2^\sharp|\cN_3)-\dim_{\cF}(\cN_1^\sharp|\cN_3)&\overset{\eqref{E-amenable extension1}}\le
\frac{\dim(\pi_{\cW\cM_1}(\cW\cM_2)|\cN_3/\cW\cM_1)}{\dim(\cW)}+\varepsilon\\
&\le \dim_{\cF}(\pi(\cN_2^\sharp)|\cN_3/\cN_1)+4\varepsilon.
\end{align*}
Therefore
\begin{align*}
\dim_{\cF}(\cN_2^\sharp|\cN_3)&\le \dim_{\cF}(\cN_1^\sharp+\cN_2^\sharp|\cN_3)\\
&\le \dim_{\cF}(\cN_1^\sharp|\cN_3)+\dim_{\cF}(\pi(\cN_2^\sharp)|\cN_3/\cN_1)+4\varepsilon\\
&\le \dim_{\cF}(\cN_1|\cN_3)+\dim_{\cF}(\pi(\cN_2^\sharp)|\cN_3/\cN_1)+4\varepsilon.
\end{align*}
Letting $\varepsilon\to 0$ we obtain \eqref{E-amenable extension3} as desired.
Therefore $\dim_{\cF}(\cdot |\cdot)$ is a bivariant Sylvester module rank function for $S$. By Theorem~\ref{T-infimum for matrix} we have $\rk_{\cF}(A)=\dim_{\cF}(S^nA|S^m)$ for every $A\in M_{n, m}(S)$. Thus $\dim_{\cF}(\cdot|\cdot)$ corresponds to $\rk_{\cF}$.
\end{proof}

\begin{remark} \label{R-length function}
A {\it normalized length function} for $R$ \cite{NR} is an $\Rb_{\ge 0}\cup \{+\infty\}$-valued function $\cM\mapsto \rL(\cM)$ on the class of all left $R$-modules satisfying the following conditions:
\begin{enumerate}
\item $\rL(\{0\})=0$ and $\rL(R)=1$.
\item $\rL(\cM)=\sup_{\cM'}\rL(\cM')$ for $\cM'$ ranging over finitely generated $R$-submodules of $\cM$.
\item For any short exact sequence $0\rightarrow \cM_1\rightarrow \cM_2\rightarrow \cM_3\rightarrow 0$ of left $R$-modules, one has $\rL(\cM_2)=\rL(\cM_1)+\rL(\cM_3)$.
\end{enumerate}
The normalized length functions for $R$ are exactly the bivariant Sylvester module rank functions $\dim(\cdot|\cdot)$ for $R$ satisfying that $\dim(\cM_1|\cM_2)=\dim(\cM_1)$ for all left $R$-modules $\cM_1\subseteq \cM_2$, via $\rL(\cdot)=\dim(\cdot)$ \cite[Proposition 4.2]{Li19}.
When $\dim(\cdot|\cdot)$ is actually a normalized length function for $R$, clearly $\dim_{\cF}(\cdot|\cdot)$ in Theorem~\ref{T-amenable extension} is also a normalized length function for $S$. When $S$ is the twisted crossed product $R*\Gamma$ for an amenable group $\Gamma$ and $\dim(\cdot|\cdot)$ is a normalized length function for $R$, the normalized length function $\dim_{\cF}(\cdot|\cdot)$ has been constructed in the literature already, first by Elek for finitely generated left $S$-modules in the case $R$ is an integral domain with the fractional field $\Kb$ and $\dim(\cM)=\dim_\Kb(\Kb\otimes_R \cM)$ for left $R$-modules $\cM$ and $S$ is the group ring $R\Gamma$ \cite{Elek03c}, then by Li and Liang in the case $S$ is the group ring $R\Gamma$ \cite[Remark 3.16]{LL}, and finally by Virili for the general case of a twisted crossed product \cite[Theorem B]{Virili}.
\end{remark}

\section{Sylvester rank function for field extensions} \label{S-field extension}

In this section we study the field extension in Example~\ref{E-field extension1} in more detail. We show that $\rk_{\cF}$ behaves well with respect to composition of field extensions, and that it extends the construction of Jaikin-Zapirain in \cite{JZ19}.

\subsection{Composition} \label{SS-composition}

Let $\Kb$  be a field and let $\Eb$ be a field containing $\Kb$. Let $R$ be a $\Kb$-algebra with a Sylvester matrix rank function $\rk$ and the corresponding bivariant Sylvester module rank function $\dim(\cdot|\cdot)$. Then $\Eb\otimes_{\Kb}R$ is a right amenable normal extension of $R$ preserving $\rk$ with $\cF$ and $\cU$ as in Example~\ref{E-field extension1}. Thus we have the Sylvester matrix rank function $\rk_{\cF}$ for $\Eb\otimes_{\Kb} R$, and we shall denote it by $\rk_{\Eb/\Kb}$ in this subsection. We denote the corresponding bivariant Sylvester module rank function by $\dim_{\Eb/\Kb}(\cdot|\cdot)$.

Let $\Eb'$ be a field containing $\Eb$. Then we have the Sylvester matrix rank function $(\rk_{\Eb/\Kb})_{\Eb'/\Eb}$ for $\Eb'\otimes_{\Eb} (\Eb\otimes_{\Kb}R)=\Eb'\otimes_{\Kb}R$ constructed out of  $\rk_{\Eb/\Kb}$, and we shall denote it by $\rk_{\Eb'/\Eb}$. We denote the corresponding bivariant Sylvester module rank function by $\dim_{\Eb'/\Eb}(\cdot|\cdot)$.

Treating $\Eb'$ as an extension of $\Kb$, we also have the Sylvester matrix rank function $\rk_{\Eb'/\Kb}$ for $\Eb'\otimes_{\Kb}R$ constructed out of $\rk$.
 We denote the corresponding bivariant Sylvester module rank function by $\dim_{\Eb'/\Kb}(\cdot|\cdot)$.

\begin{proposition} \label{P-field extension composition}
We have $\rk_{\Eb'/\Eb}=\rk_{\Eb'/\Kb}$.
\end{proposition}
\begin{proof} It suffices to show $\dim_{\Eb'/\Eb}(\cN_1|\cN_2)=\dim_{\Eb'/\Kb}(\cN_1|\cN_2)$ for all left $\Eb'\otimes_{\Kb}R$-modules $\cN_1\subseteq \cN_2$ with $\cN_1$ finitely generated.
Take a finitely generated $R$-submodule $\cM_1$ of $\cN_1$ generating $\cN_1$ as an $\Eb'\otimes_{\Kb}R$-module. Let $\varepsilon>0$. Then there are a  finite-dimensional $\Kb$-linear subspace $V$ of $\Eb'$ containing $1_{\Kb}$ and $\delta>0$ such that for any
nonzero finite-dimensional $\Kb$-linear subspace $W$ of $\Eb'$ satisfying $\dim_{\Kb}(WV)\le (1+\delta)\dim_{\Kb}(W)$ we have
\begin{align} \label{E-field extension composition1}
\dim_{\Eb'/\Kb}(\cN_1|\cN_2)\le \frac{\dim((W\otimes_{\Kb} R)\cM_1|\cN_2)}{\dim_{\Kb}(W)}\le \dim_{\Eb'/\Kb}(\cN_1|\cN_2)+\varepsilon.
\end{align}

Take $\eta>0$ with $(1+\eta)^2<1+\delta$.
Note that $(\Eb\otimes_{\Kb}R)\cM_1$ is a finitely generated $\Eb\otimes_{\Kb}R$-submodule of $\cN_1$ generating $\cN_1$ as an $\Eb'\otimes_{\Kb}R$-module, and $\Eb V$ is a nonzero finite-dimensional $\Eb$-linear subspace of $\Eb'$. Also note that
$U\otimes_\Eb(\Eb\otimes_\Kb R)=U\otimes_\Kb R$ for every $\Eb$-linear subspace $U$ of $\Eb'$.
Thus there is some
nonzero finite-dimensional $\Eb$-linear subspace $U$ of $\Eb'$ satisfying $\dim_{\Eb}(U\Eb V)<(1+\eta)\dim_{\Eb}(U)$ such that
\begin{align} \label{E-field extension composition2}
  \dim_{\Eb'/\Eb}(\cN_1|\cN_2)\le \frac{\dim_{\Eb/\Kb}((U\otimes_\Kb R)(\Eb\otimes_{\Kb} R)\cM_1|\cN_2)}{\dim_{\Eb}(U)}\le \dim_{\Eb'/\Eb}(\cN_1|\cN_2)+\varepsilon.
  \end{align}

Take an $\Eb$-basis $B$ for $U$ and an $\Eb$-basis $B'$ for $U\Eb V$ with $B\subseteq B'$. Then there is some nonzero finite-dimensional $\Kb$-linear subspace $V'$ of $\Eb$
with $1_{\Kb}\in V'$ and $\sum_{b\in B}bV\subseteq \sum_{b'\in B'}b'V'$.
Note that $\sum_{b\in B}(b\otimes_{\Kb} 1_R)\cM_1$ is a finitely generated $R$-submodule of $(U\otimes_\Kb R)(\Eb\otimes_{\Kb}R)\cM_1$ generating $(U\otimes_\Kb R)(\Eb\otimes_{\Kb}R)\cM_1$ as an $\Eb\otimes_{\Kb}R$-module.
Thus there is some
nonzero finite-dimensional $\Kb$-linear subspace $Z$ of $\Eb$ satisfying $\dim_{\Kb}(ZV')<(1+\eta)\dim_{\Kb}(Z)$ such that
\begin{align} \label{E-field extension composition3}
\dim_{\Eb/\Kb}((U\otimes_\Kb R)(\Eb\otimes_{\Kb}R)\cM_1|\cN_2)&\le \frac{\dim((Z\otimes_{\Kb}R)\sum_{b\in B}(b\otimes_{\Kb} 1_R)\cM_1|\cN_2)}{\dim_\Kb(Z)}\\
\nonumber &\hspace*{10mm} \le \dim_{\Eb/\Kb}((U\otimes_\Kb R)(\Eb\otimes_{\Kb}R)\cM_1|\cN_2)+\varepsilon.
\end{align}

Set $W=\sum_{b\in B}Zb$. Then
$$\dim_{\Kb}(W)=|B|\dim_{\Kb}(Z)=\dim_{\Eb}(U)\dim_{\Kb}(Z).$$
Thus
\begin{align*}
\dim_{\Kb}(WV)&=\dim_{\Kb}(\sum_{b\in B}ZbV)\\
&\le \dim_{\Kb}(\sum_{b'\in B'}Zb'V')\\
&=|B'|\dim_{\Kb}(ZV')\\
&= \dim_{\Eb}(U\Eb V)\dim_{\Kb}(ZV')\\
&\le (1+\eta)\dim_{\Eb}(U)(1+\eta)\dim_{\Kb}(Z)\\
&=(1+\eta)^2\dim_{\Kb}(W)\le (1+\delta)\dim_{\Kb}(W).
\end{align*}
Therefore \eqref{E-field extension composition1} holds.
Note that $(W\otimes_{\Kb}R)\cM_1=(Z\otimes_{\Kb}R)\sum_{b\in B}(b\otimes_{\Kb}1_R)\cM_1$.
Now we have
\begin{align*}
\dim_{\Eb'/\Eb}(\cN_1|\cN_2)&\overset{\eqref{E-field extension composition2}}\le \frac{\dim_{\Eb/\Kb}((U\otimes_\Kb R)(\Eb\otimes_{\Kb}R) \cM_1|\cN_2)}{\dim_{\Eb}(U)}\\
&\overset{\eqref{E-field extension composition3}}\le \frac{\dim((Z\otimes_{\Kb}R)\sum_{b\in B}(b\otimes_{\Kb}1_R)\cM_1|\cN_2)}{\dim_{\Kb}(Z)\dim_{\Eb}(U)}\\
&\overset{\eqref{E-field extension composition3}}\le \frac{\dim_{\Eb/\Kb}((U\otimes_\Kb R)(\Eb\otimes_{\Kb}R) \cM_1|\cN_2)}{\dim_{\Eb}(U)}+\varepsilon\\
&\overset{\eqref{E-field extension composition2}}\le \dim_{\Eb'/\Eb}(\cN_1|\cN_2)+2\varepsilon,
\end{align*}
and hence
\begin{align} \label{E-field extension composition4}
 \dim_{\Eb'/\Eb}(\cN_1|\cN_2)\le \frac{\dim((W\otimes_{\Kb}R)\cM_1|\cN_2)}{\dim_{\Kb}(W)}\le \dim_{\Eb'/\Eb}(\cN_1|\cN_2)+2\varepsilon.
 \end{align}
 Comparing \eqref{E-field extension composition1} and \eqref{E-field extension composition4} we get
 $$ |\dim_{\Eb'/\Eb}(\cN_1|\cN_2)-\dim_{\Eb'/\Kb}(\cN_1|\cN_2)|\le 2\varepsilon. $$
 Letting $\varepsilon\to 0$ we obtain $\dim_{\Eb'/\Eb}(\cN_1|\cN_2)=\dim_{\Eb'/\Kb}(\cN_1|\cN_2)$.
\end{proof}

\subsection{Comparison with construction of Jaikin-Zapirain} \label{SS-comparison}

Let $\Kb$ be a field and $\Eb$ a field containing $\Kb$. Let $R$ be a unital $\Kb$-algebra with a Sylvester matrix rank function $\rk$. Then we have
$\cF$ in Example~\ref{E-field extension}, and the Sylvester matrix rank function $\rk_{\cF}$ for $\Eb\otimes_\Kb R$.

When $\Eb$ is an algebraic extension of $\Kb$, Jaikin-Zapirain constructed a Sylvester matrix rank function $\rk_{\Eb\otimes_\Kb R}$ for $\Eb\otimes_\Kb R$ extending $\rk$ \cite[Section 7.5]{JZ19}.
Let $A\in M_{n, m}(\Eb\otimes_\Kb R)$. Then there is a finite subextension $\Eb_0/\Kb$ of $\Eb/\Kb$ such that $A\in M_{n, m}(\Eb_0\otimes_{\Kb} R)$. Take an $R$-basis $w_1, \dots, w_l$ of $(\Eb_0\otimes_\Kb R)^n$ as a left $R$-module and an $R$-basis $\tilde{w}_1, \dots, \tilde{w}_p$ of $(\Eb_0\otimes_{\Kb} R)^m$ as a left $R$-module. One has
\begin{eqnarray*}
\left[\begin{matrix} w_1A \\ \vdots \\ w_lA \end{matrix}\right]=B\left[\begin{matrix} \tilde{w}_1\\ \vdots\\ \tilde{w}_p \end{matrix}\right]
 \end{eqnarray*}
for some $B\in M_{l, p}(R)$. Then $\rk_{\Eb\otimes_{\Kb} R}(A)$ is defined by
$$\rk_{\Eb\otimes_{\Kb} R}(A)=\frac{\rk(B)}{\dim_{\Kb}(\Eb_0)},$$
and does not depend on the choice of $\Eb_0$.

\begin{proposition} \label{P-algebraic}
We have $\rk_{\Eb\otimes_{\Kb} R}=\rk_{\cF}$.
\end{proposition}
\begin{proof} Let $A\in M_{n, m}(\Eb\otimes_\Kb R)$. For any finite-dimensional $\Kb$-linear subspace $V$ of $\Eb$, take a finite subextension $\Eb_0/\Kb$ of $\Eb/\Kb$ such that $A\in M_{n, m}(\Eb_0\otimes_{\Kb} R)$ and $V\subseteq \Eb_0$. Then $\cW:=\Eb_0\otimes_{\Kb} R$ is in $\hat{\cF}$ and is $(V\otimes_{\Kb} R, \varepsilon)$-invariant for every $\varepsilon>0$. Clearly
$$ \rk_{\Eb\otimes_{\Kb} R}(A)=\frac{\rk_\cW(A)}{\dim(\cW)}.$$
From Theorem~\ref{T-infimum for matrix} we conclude $\rk_{\Eb\otimes_{\Kb} R}(A)=\rk_{\cF}(A)$.
\end{proof}

Next we consider the case $\Eb=\Kb(t)$ is a transcendental extension of $\Kb$. Let $f$ be  a monic polynomial in $\Kb[t]$ of degree $d$. Denote by $\pi_f$ the quotient homomorphism $R[t]\rightarrow R[t]/R[t]f$.  Note that $R[t]/R[t]f$ is a free left $R$-module with a basis $t^j+R[t]f$ for $j=0, 1, \dots, d-1$.
Let $w=(w_1, \dots, w_d)$ be an $R$-basis of $R[t]/R[t]f$ as left $R$-module.  Then right multiplication by $\pi_f(a)$ for $a\in R[t]$ induces a unital ring homomorphism $\psi_{f, w}: R[t]\rightarrow M_d(R)$. Explicitly,
\begin{eqnarray*}
\left[\begin{matrix} w_1 \\ \vdots \\ w_d \end{matrix}\right]\pi_f(a)=\psi_{f, w}(a)\left[\begin{matrix} w_1\\ \vdots\\ w_d \end{matrix}\right]
 \end{eqnarray*}
 for all $a\in R[t]$. Denote by $\rk_f$ the induced Sylvester matrix rank function on $R[t]$, i.e.
 $$\rk_f(A)=\frac{1}{d}\rk(\psi_{f, w}(A))$$
 for all $A\in M_{n, m}(R[t])$. Note that if we choose another $R$-basis $v=(v_1, \dots, v_d)$ of $R[t]/R[t]f$, then there is some invertible $B\in M_d(R)$ such that
 \begin{eqnarray*}
\left[\begin{matrix} w_1 \\ \vdots \\ w_d \end{matrix}\right]=B\left[\begin{matrix} v_1\\ \vdots\\ v_d \end{matrix}\right],
 \end{eqnarray*}
 whence $\psi_{f, w}(a)=B\psi_{f, v}(a)B^{-1}$ for all $a\in R[t]$ and thus $\rk_f$ does not depend on the choice of $w$.

 In an earlier version of this paper the following lemma is stated only in the case $f_i=t^i$ for all $i$. We are grateful to the referee for pointing out that the proof actually works in general situation.

\begin{lemma} \label{L-monic}
Let $\{f_i\}_{i\in \Nb}$ be a sequence of monic polynomials in $\Kb[t]$ with $\lim_{i\to \infty}\deg(f_i)=\infty$. Then
$$\rk_{\cF}(A)=\lim_{i\to \infty}\rk_{f_i}(A)$$
for every $A\in M_{n, m}(R[t])$.
\end{lemma}
\begin{proof} Let $A\in M_{n, m}(R[t])$.
For each $i\in \Nb$, put $d_i=\deg(f_i)$, $V_i={\rm span}_{\Kb}(1, t, \dots, t^{i-1})\subseteq \Kb[t]$ and $\cW_i=V_i\otimes_{\Kb}R\in \hat{\cF}(\Kb(t)\otimes_{\Kb}R)$. Take $p\in \Nb$ such that $A\in M_{n, m}(\cW_{p+1})$. Then $(\cW_{d_i})^nA\subseteq (\cW_{d_i+p})^m$ for every $i\in \Nb$. Take an $R$-basis $\bar{w}_{i, 1}, \dots, \bar{w}_{i, d_i n}$ of $(\cW_{d_i})^n$ as  a left $R$-module and an $R$-basis $\hat{w}_{i, 1}, \dots, \hat{w}_{i, (d_i+p)m}$ of $(\cW_{d_i+p})^m$. One has
\begin{eqnarray*}
\left[\begin{matrix} \bar{w}_{i, 1}A \\ \vdots \\ \bar{w}_{i, d_i n}A \end{matrix}\right]=B_i\left[\begin{matrix} \hat{w}_{i, 1}\\ \vdots\\ \hat{w}_{i, (d_i+p)m} \end{matrix}\right]
 \end{eqnarray*}
for some $B_i\in M_{d_i n, (d_i+p)m}(R)$. Then $\rk_{\cW_{d_i}}(A)=\rk(B_i)$ and $\dim(\cW_{d_i})=d_i$. Also take an $R$-basis $w_{i, 1}, \dots, w_{i, d_i n}$ of $(R[t]/R[t]f_i)^n$ as  a left $R$-module and an $R$-basis $\tilde{w}_{i, 1}, \dots, \tilde{w}_{i, d_i m}$ of $(R[t]/R[t]f_i)^m$. One has
\begin{eqnarray*}
\left[\begin{matrix} w_{i, 1}A \\ \vdots \\ w_{i, d_i n}A \end{matrix}\right]=\tilde{B}_i\left[\begin{matrix} \tilde{w}_{i, 1}\\ \vdots\\ \tilde{w}_{i, d_i m} \end{matrix}\right]
 \end{eqnarray*}
for some $\tilde{B}_i\in M_{d_i n, d_i m}(R)$. Then $\rk_{f_i}(A)=\frac{1}{d_i}\rk(\tilde{B}_i)$.

When $d_i>p$,  if we choose the above four bases the most natural ones, then
$$\tilde{B}_i=\left[\begin{matrix} \bar{B}_i \\ \tilde{C}_i \end{matrix}\right] \mbox{ and } B_i=\left[\begin{matrix} \bar{B}_i & \\ C_i & D_i \end{matrix}\right]$$
for some $\bar{B}_i\in M_{(d_i-p)n, d_i m}(R), \tilde{C}_i, C_i\in M_{pn, d_i m}(R)$, and $D_i\in M_{pn, pm}(R)$, and hence
$$ |\rk(\tilde{B}_i)-\rk(B_i)|\le |\rk(\tilde{B}_i)-\rk(\bar{B}_i)|+|\rk(\bar{B}_i)-\rk(B_i)|\le pn+pn=2pn.$$
Since $d_i\to \infty$ as $i\to \infty$, we get
$$\lim_{i\to \infty}\big(\rk_{f_i}(A)-\frac{\rk_{\cW_{d_i}}(A)}{\dim(\cW_{d_i})}\big)=\lim_{i\to \infty}\frac{\rk(\tilde{B}_i)-\rk(B_i)}{d_i}=0.$$

Note that every $\cV\in \cF$ is contained in $ (h\otimes_\Kb 1_R) \cW_j$ for some $j\in \Nb$ and invertible $h\in \Kb(t)$. It is easily checked that
for every $\cV\in \cF$ and $\varepsilon>0$, when $i\in \Nb$ is large enough, $\cW_i$ is $(\cV, \varepsilon)$-invariant.
Since $d_i\to \infty$ as $i\to \infty$, it follows that for every $\cV\in \cF$ and $\varepsilon>0$, when $i\in \Nb$ is large enough, $\cW_{d_i}$ is $(\cV, \varepsilon)$-invariant.
From Theorem~\ref{T-infimum for matrix} we conclude
$$\rk_{\cF}(A)=\lim_{i\to \infty}\frac{\rk_{\cW_{d_i}}(A)}{\dim(\cW_{d_i})}=\lim_{i\to \infty}\rk_{f_i}(A).$$
\end{proof}

A unital ring $R'$ is called {\it von Neumann regular} if for any $a\in R'$ there is some $b\in R'$ such that $aba=a$ \cite{Goodearl} \cite[page 61]{Lam01}. The Sylvester matrix rank function $\rk$ for  $R$ is called {\it regular} if there are a ring homomorphism $\varphi$ from $R$ to some von Neumann regular ring $R'$ and a Sylvester matrix rank function $\rk'$ for $R'$ such that $\rk(A)=\rk'(\varphi(A))$ for all $A\in M_{n, m}(R)$.

When $\Eb=\Kb(t)$ is a transcendental extension of $\Kb$ and $\rk$ is regular, Jaikin-Zapirain constructed a Sylvester matrix rank function $\rk_{\Kb(t)\otimes_\Kb R}$ for $\Kb(t)\otimes_\Kb R$ extending $\rk$ \cite[Section 7.2]{JZ19}.
Let $A\in M_{n, m}(R[t])$.
Then $\rk_{\Kb(t)\otimes_\Kb R}(A)$ is defined by
$$\rk_{\Kb(t)\otimes_\Kb R}(A)=\lim_{i\to \infty}\rk_{t^i}(A).$$

\begin{proposition} \label{P-transcendental}
We have $\rk_{\Kb(t)\otimes_\Kb R}=\rk_{\cF}$.
\end{proposition}
\begin{proof}
Since each element of $M_{n, m}(\Kb(t)\otimes_{\Kb} R)$ is of the form $(f\otimes_\Kb 1_R)A$ for some $A\in M_{n, m}(R[t])$ and invertible $f\in \Kb(t)$,  it suffices to show $\rk_{\Kb(t)\otimes_\Kb R}(A)=\rk_{\cF}(A)$ for $A\in M_{n, m}(R[t])$. This follows from Lemma~\ref{L-monic}.
\end{proof}

Propositions~\ref{P-algebraic} and \ref{P-transcendental} tell us that our construction of $\rk_{\cF}$ extends the construction of Jaikin-Zapirain for $\rk_{\Eb\otimes_{\Kb} R}$ (when $\Eb/\Kb$ is algebraic) and $\rk_{\Kb(t)\otimes_\Kb R}$ in \cite{JZ19}.
This answers his question \cite[Question 8.8]{JZ17} affirmatively.

We remark that Corollary 7.8 and Proposition 7.13 of \cite{JZ19} are consequences of Theorem~\ref{T-continuity} and Propositions~\ref{P-transcendental} and \ref{P-algebraic}.

Jaikin-Zapirain proved Proposition~\ref{P-limit rank} below in the case $\rk$ is regular \cite[Corollary 7.9]{JZ19}, in terms of $\rk_{\Kb(t)\otimes_\Kb R}$. In an earlier version of this paper Proposition \ref{P-limit rank} is stated as a question. Lemma~\ref{L-average} and Proposition~\ref{P-limit rank} are due to the referee.

\begin{lemma} \label{L-average}
Let $d\in \Nb$ and  let $x_1, \dots, x_d$ be distinct elements in $\Kb$. For each $1\le i\le d$ denote by $\pi_i$ the quotient map $R[t]\rightarrow R$ sending $h(t)$ to $h(x_i)$.  Put $f(t)=\prod_{i=1}^d(t-x_i)\in \Kb[t]$. Then
$$ \rk_f(A)=\frac{1}{d}\sum_{i=1}^d\rk(\pi_i(A))$$
for every $A\in M_{n, m}(R[t])$.
\end{lemma}
\begin{proof} We have a natural isomorphism of $\Kb$-algebras:
$$ \varphi': \Kb[t]/\Kb[t]f\longrightarrow \bigoplus_{i=1}^d\Kb[t]/\Kb[t](t-x_i)\cong \bigoplus_{i=1}^d \Kb$$
given by $(\varphi'(h+\Kb[t]f))_i=h(x_i)$ for $h\in \Kb[t]$ and $1\le i\le d$.
Taking tensor product with $R$, we obtain an isomorphism of $\Kb$-algebras:
$$ \varphi: R[t]/R[t]f\longrightarrow \bigoplus_{i=1}^d R[t]/R[t](t-x_i)\cong \bigoplus_{i=1}^d R$$
given  by $(\varphi(h+R[t]f))_i=h(x_i)=\pi_i(h)$ for $h\in R[t]$ and $1\le i\le d$.
Then $(\varphi\circ \pi_f)(h)=(\pi_1(h), \dots, \pi_d(h))$ for every $h\in R[t]$.

Let $(e_1, \dots, e_d)$ be the standard $R$-basis of the left $R$-module $\bigoplus_{i=1}^d R$, i.e. $(e_j)_i=\delta_{i, j}$ for all $1\le j, i\le d$.
For each $1\le j\le d$ put $w_j=\varphi^{-1}(e_j)\in R[t]/R[t]f$.
Then $w=(w_1, \dots, w_d)$ is an $R$-basis of the left $R$-module $R[t]/R[t]f$. For each $h\in R[t]$ we have
\begin{eqnarray*}
\left[\begin{matrix} e_1 \\ \vdots \\ e_d \end{matrix}\right]\varphi(\pi_f(h))=\left[\begin{matrix} \pi_1(h)e_1 \\ \vdots \\ \pi_d(h)e_d \end{matrix}\right]=\left[\begin{matrix} \pi_1(h) & & \\ & \ddots & \\ & & \pi_d(h) \end{matrix}\right]\cdot \left[\begin{matrix} e_1\\ \vdots\\ e_d \end{matrix}\right].
 \end{eqnarray*}
Applying $\varphi^{-1}$ on both sides, we get
\begin{eqnarray*}
\left[\begin{matrix} w_1 \\ \vdots \\ w_d \end{matrix}\right]\pi_f(h)=\left[\begin{matrix} \pi_1(h) & & \\ & \ddots & \\ & & \pi_d(h) \end{matrix}\right]\cdot \left[\begin{matrix} w_1\\ \vdots\\ w_d \end{matrix}\right],
 \end{eqnarray*}
 whence
 \begin{eqnarray*}
 \psi_{f, w}(h)=\left[\begin{matrix} \pi_1(h) & & \\ & \ddots & \\ & & \pi_d(h) \end{matrix}\right]
 \end{eqnarray*}
 for every $h\in R[t]$. Then for every $A\in M_{n, m}(R[t])$, we can obtain
 $$\left[\begin{matrix} \pi_1(A) & & \\ & \ddots & \\ & & \pi_d(A) \end{matrix}\right]$$
 from  $ \psi_{f, w}(A)$ via certain row exchanges and column exchanges, therefore
\begin{align*}
\rk_f(A)=\frac{1}{d}\rk(\psi_{f, w}(A))=\frac{1}{d}\rk \left[\begin{matrix} \pi_1(A) & & \\ & \ddots & \\ & & \pi_d(A) \end{matrix}\right] =\frac{1}{d}\sum_{j=1}^d\rk(\pi_j(A)).
\end{align*}
\end{proof}

\begin{proposition} \label{P-limit rank}
 Let $\{x_i\}_{i\in \Nb}$ be a sequence of distinct elements in $\Kb$. For each $i\in \Nb$ denote by $\pi_i$ the quotient map $R[t]\rightarrow R$ sending $f(t)$ to $f(x_i)$. Then
 $$\rk_{\cF}(A)=\lim_{i\to \infty}\rk(\pi_i(A))$$
 for every $A\in M_{n, m}(R[t])$.
\end{proposition}
\begin{proof} Let $A\in M_{n, m}(R)$. Since $0\le \rk(\pi_i(A))\le n$ for every $i$, it suffices to show that every convergent subsequence of $\{\rk(\pi_i(A))\}_{i\in \Nb}$ converges to $\rk_{\cF}(A)$. Thus passing to a convergent subsequence if necessary, we may assume that
$\rk(\pi_i(A))$ converges to some $L$ as $i\to \infty$.

For each $d\in \Nb$, put $f_d=\prod_{i=1}^d(t-x_i)\in \Kb[t]$. Then $\deg(f_d)\to \infty$ as $d\to \infty$, thus by Lemma~\ref{L-monic} we have\
$$ \rk_{\cF}(A)=\lim_{d\to \infty}\rk_{f_d}(A).$$
Since $x_i\neq x_j$ for $i\neq j$, by Lemma~\ref{L-average} we have
$$\rk_{f_d}(A)=\frac{1}{d}\sum_{i=1}^d\rk_i(A)$$
for each $d\in \Nb$. Then $\rk_{f_d}(A)\to L$ as $d\to \infty$. Therefore $L=\rk_{\cF}(A)$.
\end{proof}

\section{Extension for twisted crossed product $C^*$-algebras} \label{S-trace}

In this section we prove Theorem~\ref{T-same}.

Let $(\alpha, u)$ be a twisted action of a discrete group $\Gamma$ on a unital $C^*$-algebra $\sA$, and let $\sA\rtimes_{\alpha, u}\Gamma$ be the maximal twisted crossed product $C^*$-algebra as in Section~\ref{SS-twisted}. We may think of $\sA$ as a subalgebra of $\sA\rtimes_{\alpha, u}\Gamma$  via the embedding $a\mapsto a \overline{e_\Gamma}$. We say that {\it $(\alpha, u)$ preserves a tracial state $\tr$} of $\sA$ if $\tr$ is  preserved by $\alpha_s$ for every $s\in \Gamma$. Fix a tracial state $\tr$ of $\sA$ preserved by $(\alpha, u)$.

We observe first that $\tr\circ \cE$ is a tracial state of $\sA\rtimes_{\alpha, u}\Gamma$ extending $\tr$.

\begin{lemma} \label{L-trace extension}
$\tr_{\alpha, u}:=\tr\circ \cE$ is a tracial state of $\sA\rtimes_{\alpha, u}\Gamma$ extending $\tr$.
\end{lemma}
\begin{proof} Clearly $\tr_{\alpha, u}$ extends $\tr$. Since $\cE$ and $\tr$ are unital positive linear maps, so is $\tr_{\alpha, u}$. That is, $\tr_{\alpha, u}$ is a state.

For any $a, b\in \sA$ and $s\in \Gamma$, we have
\begin{align*}
\tr_{\alpha, u}(a\bar{s} \cdot b\overline{s^{-1}})&=\tr_{\alpha, u}(a \alpha_s(b)u_{s, s^{-1}} \overline{e_\Gamma})\\
&=\tr(a \alpha_s(b)u_{s, s^{-1}})\\
&=\tr(\alpha_{s^{-1}}(a)\alpha_{s^{-1}}(\alpha_s(b))\alpha_{s^{-1}}(u_{s, s^{-1}}))\\
&=\tr(\alpha_{s^{-1}}(a)u_{s^{-1}, s}bu_{s^{-1}, s}^*u_{s^{-1}, s})\\
&=\tr(b\alpha_{s^{-1}}(a)u_{s^{-1}, s})\\
&=\tr_{\alpha, u}(b\alpha_{s^{-1}}(a)u_{s^{-1}, s} \overline{e_\Gamma})\\
&=\tr_{\alpha, u}(b\overline{s^{-1}}\cdot a\bar{s}).
\end{align*}
Since $\tr_{\alpha, u}(a\bar{s})=0$ for all $a\in \sA$ and $s\in \Gamma\setminus \{e_\Gamma\}$, it follows that $\tr_{\alpha, u}(fg)=\tr_{\alpha, u}(gf)$ for all $f, g\in \sA*\Gamma$. As $\tr_{\alpha, u}$ is continuous and $\sA*\Gamma$ is dense in $\sA\rtimes_{\alpha, u}\Gamma$, we conclude that $\tr_{\alpha, u}(fg)=\tr_{\alpha, u}(gf)$ for all $f, g\in \sA\rtimes_{\alpha, u}\Gamma$.
\end{proof}

\begin{lemma} \label{L-unitary}
Let $F$ be a nonempty finite subset of $\Gamma$. Put $\cW=\sum_{s\in F}\sA \bar{s}\subseteq \sA*\Gamma$.
There is a unitary map $\overline{\cW+I_{\tr_{\alpha, u}}}\rightarrow \bigoplus_{s\in F}L^2(\sA, \tr)$ sending $a\bar{s}+I_{\tr_{\alpha, u}}$ for $a\in \sA$ and $s\in F$ to the vector being $\alpha_{s^{-1}}(a)u_{s^{-1}, s}+I_\tr$ at $s$ and $0$ everywhere else.
\end{lemma}
\begin{proof}
For any $a, b\in \sA$ and $s\in F$, we have
\begin{align*}
\left<a \bar{s}+I_{\tr_{\alpha, u}}, b\bar{s}+I_{\tr_{\alpha, u}}\right>&=\tr_{\alpha, u}((b\bar{s})^*(a\bar{s}))\\
&=\tr_{\alpha, u}(u_{s^{-1}, s}^*\alpha_{s^{-1}}(b^*)\overline{s^{-1}}a\bar{s})\\
&=\tr_{\alpha, u}(u_{s^{-1}, s}^*\alpha_{s^{-1}}(b^*)\alpha_{s^{-1}}(a)u_{s^{-1}, s}\overline{e_\Gamma})\\
&=\tr(u_{s^{-1}, s}^*\alpha_{s^{-1}}(b^*)\alpha_{s^{-1}}(a)u_{s^{-1}, s})\\
&=\left<\alpha_{s^{-1}}(a)u_{s^{-1}, s}+I_\tr, \alpha_{s^{-1}}(b)u_{s^{-1}, s}+I_\tr \right>.
\end{align*}
For any $a, b\in \sA$ and distinct $s, t\in F$, we have $\left<a \bar{s}+I_{\tr_{\alpha, u}}, b\bar{t}+I_{\tr_{\alpha, u}}\right>=0$. Now the lemma follows easily.
\end{proof}

We are ready to prove Theorem~\ref{T-same}.

\begin{proof}[Proof of Theorem~\ref{T-same}]
For any vector space $V$ and $r\in \Nb$, any $v\in V$, and any $1\le k\le r$, we write $v\otimes \delta_k$ for the vector in $V^r$ being $v$ at the $k$-th column and $0$ everywhere else, and write $v\otimes \delta^k$ for the vector in $V^{r\times 1}$ being $v$ at the $k$-th row and $0$ everywhere else.

For each nonempty finite subset $F$ of $\Gamma$, put $\cW_F=\sum_{s\in F}\sA \bar{s}\in \hat{\cF}$.

Our argument is similar to the proof of Elek for Linnell's analytic zero-divisor conjecture in the amenable group case \cite{Elek03b}.
Take a finite subset $K$ of $\Gamma$ containing $e_\Gamma$ such that $A_{jk}\in \cW_K$ for all $1\le j\le n$ and $1\le k\le m$. Let $F$ be a nonempty finite subset of $\Gamma$.
Take an $\sA$-basis $w_1, \dots, w_{n|F|}$ for $\cW_F^n$ consisting of $\bar{s}\otimes \delta_j$ for $s\in F$ and $1\le j\le n$, and an $\sA$-basis $\tilde{w}_1, \dots, \tilde{w}_{m|FK|}$ for $\cW_{FK}^m$ consisting of $\bar{s}\otimes \delta_k$ for $s\in FK$ and $1\le k\le m$. Then we have a matrix $B\in M_{n|F|, m|FK|}(\sA)$ determined by
\begin{eqnarray*}
\left[\begin{matrix} w_1A \\ \vdots \\ w_{n|F|}A \end{matrix}\right]=B\left[\begin{matrix} \tilde{w}_1\\ \vdots\\ \tilde{w}_{m|FK|} \end{matrix}\right],
\end{eqnarray*}
and $(\rk_\tr)_{\cW_F}(A)=\rk_\tr(B)$.

Note that
\begin{align} \label{E-basis}
(A^*w_1^*, \dots, A^*w_{n|F|}^*)=(\tilde{w}_1^*, \dots, \tilde{w}_{m|FK|}^*)B^*.
\end{align}
Here $(\bar{s}\otimes \delta_j)^*=u_{s^{-1}, s}^*\overline{s^{-1}}\otimes \delta^j$ for $s\in F$ and $1\le j\le n$, and  $(\bar{s}\otimes \delta_k)^*=u_{s^{-1}, s}^*\overline{s^{-1}}\otimes \delta^k$ for $s\in FK$ and $1\le k\le m$.
Also note that for any $a\in \sA$ and $s\in \Gamma$, if we write $u_{s^{-1}, s}^*\overline{s^{-1}}a$ as $b\overline{s^{-1}}$, then
$b=u_{s^{-1}, s}^*\alpha_{s^{-1}}(a)$ and hence
$$\alpha_s(b)u_{s, s^{-1}}=\alpha_s(u_{s^{-1}, s}^*)u_{s, s^{-1}}au_{s, s^{-1}}^*u_{s, s^{-1}}=u_{s, s^{-1}}^*u_{s, s^{-1}}a=a.$$
Denote  by $\sV_{F^{-1}, n}$ the bijective linear map $(\cW_{F^{-1}})^{n\times 1}\rightarrow \sA^{n|F|\times 1}$ sending $b\overline{s^{-1}}\otimes \delta^j$ for $b\in \sA$, $s\in F$ and $1\le j\le n$ to $\alpha_s(b)u_{s, s^{-1}}\otimes \delta^q$ such that $w_q=\bar{s}\otimes \delta_j$, and
by $\sV_{(FK)^{-1}, m}$ the bijective linear map $(\cW_{(FK)^{-1}})^{m\times 1}\rightarrow \sA^{m|FK|\times 1}$ sending $b\overline{s^{-1}}\otimes \delta^k$ for $b\in \sA$, $s\in FK$ and $1\le k\le m$ to $\alpha_s(b)u_{s, s^{-1}}\otimes \delta^p$ such that $\tilde{w}_p=\bar{s}\otimes \delta_k$.
Then
$$ \sV_{F^{-1}, n}((w_1^*, \dots, w_{n|F|}^*)y)=y$$
for all $y\in \sA^{n|F|\times 1}$, and
$$ \sV_{(FK)^{-1}, m}((\tilde{w}_1^*, \dots, \tilde{w}_{m|FK|}^*)x)=x$$
for all $x\in \sA^{m|FK|\times 1}$.
Now for every $y\in \sA^{n|F|\times 1}$ we have
\begin{align*}
A^*(w_1^*, \dots, w_{n|F|}^*)y=(A^*w_1^*, \dots, A^*w_{n|F|}^*)y\overset{\eqref{E-basis}}=(\tilde{w}_1^*, \dots, \tilde{w}_{m|FK|}^*)B^*y,
\end{align*}
and
hence
$$ \sV_{(FK)^{-1}, m}(A^*(w_1^*, \dots, w_{n|F|}^*)y)=\sV_{(FK)^{-1}, m}((\tilde{w}_1^*, \dots, \tilde{w}_{m|FK|}^*)B^*y)=B^*y.$$
Thus
\begin{align} \label{E-basis change}
 \sV_{(FK)^{-1}, m}(A^*z)=B^*\sV_{F^{-1}, n}(z)
 \end{align}
for all $z\in (\cW_{F^{-1}})^{n\times 1}$.

By Lemma~\ref{L-unitary} there is  a unitary operator $U_{F^{-1}, n}: \overline{\cW_{F^{-1}}+I_{\tr_{\alpha, u}}}^{n\times 1}\rightarrow L^2(\sA, \tr)^{n|F|\times 1}$ sending
$(b\overline{s^{-1}}+I_{\tr_{\alpha, u}})\otimes \delta^j$ for $s\in F$ and $1\le j\le n$ to $(\alpha_s(b)u_{s, s^{-1}}+I_\tr)\otimes \delta^q$ such that $w_q=\bar{s}\otimes \delta_j$, and there is  a unitary operator $U_{(FK)^{-1}, m}: \overline{\cW_{(FK)^{-1}}+I_{\tr_{\alpha, u}}}^{m\times 1}\rightarrow L^2(\sA, \tr)^{m|FK|\times 1}$ sending
$(b\overline{s^{-1}}+I_{\tr_{\alpha, u}})\otimes \delta^k$ for $s\in FK$ and $1\le k\le m$ to $(\alpha_s(b)u_{s, s^{-1}}+I_\tr)\otimes \delta^p$ such that $\tilde{w}_p=\bar{s}\otimes \delta_k$.

Put
$$H=\overline{\pi_{\tr_{\alpha, u}}(A^*)\cdot(\cW_{F^{-1}}+I_{\tr_{\alpha, u}})^{n\times 1}}\subseteq \overline{\cW_{(FK)^{-1}}+I_{\tr_{\alpha, u}}}^{m\times 1}\subseteq L^2(\sA\rtimes_{\alpha, u}\Gamma, \tr_{\alpha, u})^{m\times 1}.$$
From \eqref{E-basis change} we have
$$ U_{(FK)^{-1}, m}\pi_{\tr_{\alpha, u}}(A^*)=\pi_\tr(B^*)U_{F^{-1}, n}$$
on $\overline{\cW_{F^{-1}}+I_{\tr_{\alpha, u}}}^{n\times 1}$. Therefore
\begin{align} \label{E-kernel proj}
 U_{F^{-1}, n}(\overline{\cW_{F^{-1}}+I_{\tr_{\alpha, u}}}^{n\times 1}\cap \ker \pi_{\tr_{\alpha, u}}(A^*))=\ker \pi_\tr(B^*).
\end{align}
and
\begin{align} \label{E-image proj}
U_{(FK)^{-1},m}(H)=\overline{\pi_\tr(B^*)\cdot L^2(\sA, \tr)^{n|F|\times 1}}=\overline{\im \pi_\tr(B^*)}.
\end{align}
Denote by $C^\sharp$ the diagonal matrix in $M_{n|F|}(\sA)$ such that $C^\sharp_{qq}=u_{s, s^{-1}}$ for $1\le q\le n|F|$ and $s\in F$ satisfying $w_q=\bar{s}\otimes \delta_j$, and
by $C$ the diagonal matrix in $M_{m|FK|}(\sA)$ such that $C_{pp}=u_{s, s^{-1}}$ for $1\le p\le m|FK|$ and $s\in FK$ satisfying $\tilde{w}_p=\bar{s}\otimes \delta_k$.

Denote by $Q$ ($\tilde{Q}$ resp.)  the orthogonal projection from $L^2(\sA\rtimes_{\alpha, u}\Gamma, \tr_{\alpha, u})^{m\times 1}$ ($\overline{\cW_{(FK)^{-1}}+I_{\tr_{\alpha, u}}}^{m\times 1}$ resp.) to $H$. Then $H\subseteq \overline{\im \pi_{\tr_{\alpha, u}}(A^*)}$, and hence $Q\le P_{\overline{\im \pi_{\tr_{\alpha, u}}(A^*)}}$. Also from \eqref{E-image proj} we have $U_{(FK)^{-1}, m} \tilde{Q}U_{(FK)^{-1}, m}^*=P_{\overline{\im \pi_\tr(B^*)}}$.
Thus
\begin{align*}
\rk_{\tr_{\alpha, u}}(A^*)&=\tr''_{\alpha, u}(P_{\overline{\im \pi_{\tr_{\alpha, u}}(A^*)}})\\
&=\frac{1}{|FK|}\sum_{t\in (FK)^{-1}}\tr''_{\alpha, u}(\pi_{\tr_{\alpha, u}}(\bar{t}^*)P_{\overline{\im \pi_{\tr_{\alpha, u}}(A^*)}}\pi_{\tr_{\alpha, u}}(\bar{t}))\\
&=\frac{1}{|FK|}\sum_{t\in (FK)^{-1}}\sum_{k=1}^m\left<\pi_{\tr_{\alpha, u}}(\bar{t}^*)P_{\overline{\im \pi_{\tr_{\alpha, u}}(A^*)}}\pi_{\tr_{\alpha, u}}(\bar{t})(\xi_{\tr_{\alpha, u}}\otimes \delta^k), \xi_{\tr_{\alpha, u}}\otimes \delta^k\right>\\
&=\frac{1}{|FK|}\sum_{t\in (FK)^{-1}}\sum_{k=1}^m\left<P_{\overline{\im \pi_{\tr_{\alpha, u}}(A^*)}} ((\bar{t}+I_{\tr_{\alpha, u}})\otimes \delta^k), (\bar{t}+I_{\tr_{\alpha, u}})\otimes \delta^k\right>\\
&\ge \frac{1}{|FK|}\sum_{t\in (FK)^{-1}}\sum_{k=1}^m\left<Q ((\bar{t}+I_{\tr_{\alpha, u}})\otimes \delta^k), (\bar{t}+I_{\tr_{\alpha, u}})\otimes \delta^k\right>\\
&=\frac{1}{|FK|}\sum_{t\in (FK)^{-1}}\sum_{k=1}^m\left<\tilde{Q}((\bar{t}+I_{\tr_{\alpha, u}})\otimes \delta^k), (\bar{t}+I_{\tr_{\alpha, u}})\otimes \delta^k\right>\\
&=\frac{1}{|FK|}\sum_{t\in (FK)^{-1}}\sum_{k=1}^m\left<P_{\overline{\im \pi_\tr(B^*)}} U_{(FK)^{-1}, m}((\bar{t}+I_{\tr_{\alpha, u}})\otimes \delta^k), U_{(FK)^{-1}, m}((\bar{t}+I_{\tr_{\alpha, u}})\otimes \delta^k)\right>\\
&=\frac{1}{|FK|}\sum_{p=1}^{m |FK|}\left<P_{\overline{\im \pi_\tr(B^*)}}\pi_\tr(C) (\xi_\tr\otimes \delta^p),  \pi_\tr(C)(\xi_\tr\otimes \delta^p) \right>\\
&=\frac{1}{|FK|}\tr''(\pi_\tr(C)^*P_{\overline{\im \pi_\tr(B^*)}}\pi_\tr(C))\\
&=\frac{1}{|FK|}\tr''(P_{\overline{\im \pi_\tr(B^*)}})\\
&=\frac{1}{|FK|} \rk_\tr(B^*)\\
&\overset{\eqref{E-rank for adjoint}}=\frac{1}{|FK|} \rk_\tr(B)\\
&=\frac{|F|}{|FK|}\frac{(\rk_\tr)_{\cW_F}(A)}{\dim(\cW_F)}.
\end{align*}
Therefore
$$
\rk_{\tr_{\alpha, u}}(A^*)\ge \lim_{\cW_F}\frac{|F|}{|FK|}\frac{(\rk_\tr)_{\cW_F}(A)}{\dim(\cW_F)}=\lim_{\cW_F}\frac{(\rk_\tr)_{\cW_F}(A)}{\dim(\cW_F)}=(\rk_\tr)_{\cF}(A).$$

Denote by $Q^\sharp$ ($Q^\dag$ resp.) the orthogonal projection from $L^2(\sA\rtimes_{\alpha, u}\Gamma, \tr_{\alpha, u})^{n\times 1}$ ($\overline{\cW_{F^{-1}}+I_{\tr_{\alpha, u}}}^{n\times 1}$ resp.) to $\overline{\cW_{F^{-1}}+I_{\tr_{\alpha, u}}}^{n\times 1}\cap \ker \pi_{\tr_{\alpha, u}}(A^*)$. Then $Q^\sharp\le P_{\ker \pi_{\tr_{\alpha, u}}(A^*)}$. Also from \eqref{E-kernel proj} we have $U_{F^{-1}, n} Q^\dag U_{F^{-1}, n}^*=P_{\ker \pi_\tr(B^*)}$.
Thus
\begin{align*}
\tr''_{\alpha, u}(P_{\ker \pi_{\tr_{\alpha, u}}(A^*)})&=\frac{1}{|F|}\sum_{t\in F^{-1}}\tr''_{\alpha, u}(\pi_{\tr_{\alpha, u}}(\bar{t}^*)P_{\ker \pi_{\tr_{\alpha, u}}(A^*)}\pi_{\tr_{\alpha, u}}(\bar{t}))\\
&=\frac{1}{|F|}\sum_{t\in F^{-1}}\sum_{j=1}^n\left<\pi_{\tr_{\alpha, u}}(\bar{t}^*)P_{\ker \pi_{\tr_{\alpha, u}}(A^*)}\pi_{\tr_{\alpha, u}}(\bar{t})(\xi_{\tr_{\alpha, u}}\otimes \delta^j), \xi_{\tr_{\alpha, u}}\otimes \delta^j\right>\\
&=\frac{1}{|F|}\sum_{t\in F^{-1}}\sum_{j=1}^n\left<P_{\ker \pi_{\tr_{\alpha, u}}(A^*)}((\bar{t}+I_{\tr_{\alpha, u}})\otimes \delta^j), (\bar{t}+I_{\tr_{\alpha, u}})\otimes \delta^j\right>\\
&\ge \frac{1}{|F|}\sum_{t\in F^{-1}}\sum_{j=1}^n\left<Q^\sharp((\bar{t}+I_{\tr_{\alpha, u}})\otimes \delta^j), (\bar{t}+I_{\tr_{\alpha, u}})\otimes \delta^j\right>\\
&= \frac{1}{|F|}\sum_{t\in F^{-1}}\sum_{j=1}^n\left<Q^\dag((\bar{t}+I_{\tr_{\alpha, u}})\otimes \delta^j), (\bar{t}+I_{\tr_{\alpha, u}})\otimes \delta^j\right>\\
&= \frac{1}{|F|}\sum_{t\in F^{-1}}\sum_{j=1}^n\left<P_{\ker \pi_\tr(B^*)}U_{F^{-1}, n}((\bar{t}+I_{\tr_{\alpha, u}})\otimes \delta^j), U_{F^{-1}, n}((\bar{t}+I_{\tr_{\alpha, u}})\otimes \delta^j)\right>\\
&=\frac{1}{|F|}\sum_{q=1}^{n|F|}\left<P_{\ker \pi_\tr(B^*)}\pi_\tr(C^\sharp)(\xi_\tr\otimes \delta^q), \pi_\tr(C^\sharp)(\xi_\tr\otimes \delta^q)\right>\\
&=\frac{1}{|F|}\tr''(\pi_\tr(C^\sharp)^*P_{\ker \pi_\tr(B^*)}\pi_\tr(C^\sharp))\\
&=\frac{1}{|F|}\tr''(P_{\ker \pi_\tr(B^*)}),
\end{align*}
and hence
\begin{align*}
\rk_{\tr_{\alpha, u}}(A^*)&\overset{\eqref{E-rank for kernel}}=n-\tr''_{\alpha, u}(P_{\ker \pi_{\tr_{\alpha, u}}(A^*)})\\
&\le \frac{1}{|F|}\big(n|F|-\tr''(P_{\ker \pi_\tr(B^*)})\big)\\
&\overset{\eqref{E-rank for kernel}}=\frac{1}{|F|} \rk_\tr(B^*)\\
&\overset{\eqref{E-rank for adjoint}}=\frac{1}{|F|}\rk_\tr(B)\\
&=\frac{(\rk_\tr)_{\cW_F}(A)}{\dim(\cW_F)}.
\end{align*}
Therefore
$$\rk_{\tr_{\alpha, u}}(A^*)\le \lim_{\cW_F}\frac{(\rk_\tr)_{\cW_F}(A)}{\dim(\cW_F)}=(\rk_\tr)_{\cF}(A).$$
We conclude that
$$ \rk_{\tr_{\alpha, u}}(A)\overset{\eqref{E-rank for adjoint}}=\rk_{\tr_{\alpha, u}}(A^*)= (\rk_\tr)_{\cF}(A).$$
\end{proof}


\end{document}